\documentclass[11pt]{amsart}

\usepackage{a4wide,enumerate,color,graphicx}
\usepackage{amsmath,bm}
\usepackage{scrextend} 
\allowdisplaybreaks
\usepackage[normalem]{ulem}
\usepackage[dvipsnames]{xcolor}
\usepackage{setspace}
\usepackage{amsbsy}
\usepackage{hyperref}
\usepackage[overload]{empheq}

\newcommand{\R}{{\mathbb R}} 
\newcommand{\del}{\partial}
\newcommand{\diver}{\operatorname{div}} 

\newcommand{\dx}{\textnormal{d}x}
\newcommand{\dt}{\textnormal{d}t}
\newcommand{\ds}{\textnormal{d}s}
\newcommand{\dy}{\textnormal{d}y}
\newcommand{\dtau}{\textnormal{d}\tau}
\newcommand{\dz}{\textnormal{d}z}

\newcommand{\torus}{\mathbb{T}}
\newcommand{\T}{{\mathbb T}} 
\newcommand{\torusd}{{\mathbb{T}^3}}
\newcommand{\dX}{\textnormal{d}X}
\newcommand{\dY}{\textnormal{d}Y}
\newcommand{\dZ}{\textnormal{d}Z}
\newcommand{\dsigma}{\textnormal{d}\sigma}

\newcommand{\cc}{\mathbf{c}}
\newcommand{\bcc}{\bar{\mathbf{c}}}
\newcommand{\dd}{\mathbf{d}}
\newcommand{\eps}{\epsilon}
\newcommand{\YY}{\boldsymbol{Y}}

\newcommand{\brho}{\bar{\rho}}
\newcommand{\bu}{\bar{u}}

\numberwithin{equation}{section}
\hypersetup{
    colorlinks=true,
    linkcolor=blue,
    filecolor=blue,      
    urlcolor=blue,
}

\theoremstyle{plain}
\newtheorem{theorem}{Theorem}

\newtheorem{lemma}[theorem]{Lemma}
\newtheorem{corollary}[theorem]{Corollary}
\newtheorem{definition}[theorem]{Definition}

\begin{document}

	\title[Renormalized solutions and Uniqueness]{Renormalized solutions for the Maxwell--Stefan system with an application to uniqueness of weak solutions}

        \author[S. Georgiadis]{Stefanos Georgiadis}
        \address{Computer, Electrical and Mathematical Sciences and Engineering Division,
        King Abdullah University of Science and Technology (KAUST),
        Thuwal 23955-6900, Saudi Arabia and Institute for Analysis and Scientific Computing, Vienna University of Technology, Wiedner Hauptstra\ss e 8-10, 1040 Wien, Austria}
        \email{stefanos.georgiadis@kaust.edu.sa} 

        \author[H. Kim]{Hoyoun Kim}
        \address{Computer, Electrical and Mathematical Sciences and Engineering Division, King Abdullah University of Science and Technology (KAUST), Thuwal 23955-6900, Saudi Arabia}
        \email{hoyoun.kim@kaust.edu.sa}

        \author[A. E. Tzavaras]{Athanasios E. Tzavaras}
        \address{Computer, Electrical and Mathematical Sciences and Engineering Division, King Abdullah University of Science and Technology (KAUST), Thuwal 23955-6900, Saudi Arabia}
        \email{athanasios.tzavaras@kaust.edu.sa}
	
	%\date{\today}

\baselineskip = 12pt
	
	\begin{abstract}
	We give conditions that guarantee uniqueness  of renormalized solutions for the Maxwell-Stefan system. The proof
		is based on an identity for the evolution of the symmetrized relative entropy.
		Using the method of doubling the variables we derive the identity for two renormalized solutions and use information
		on the spectrum of the Maxwell-Stefan matrix to estimate the symmetrized relative entropy and show uniqueness.
		We then show that weak solutions for the Maxwell-Stefan system have sufficient regularity to produce renormalized solutions. 
		Combining the two results  yields a uniqueness result for weak solutions of the Maxwell-Stefan system with bounded fluxes.
    \end{abstract}

        %35D30  	Weak solutions to PDEs
        %35D99  	None of the above, but in this section
        %35Q35  	PDEs in connection with fluid mechanics
        %76N10  	Existence, uniqueness, and regularity theory for compressible fluids and gas dynamics
        %76R50  	Diffusion
        %76T30  	Three or more component flows
	
	\maketitle
	
	\tableofcontents
%----------------------------------------------------------------------------------------
%----------------------------------------------------------------------------------------
%----------------------------------------------------------------------------------------
%----------------------------------------------------------------------------------------
%----------------------------------------------------------------------------------------

\section{Introduction}
 
The Maxwell-Stefan system describes diffusive phenomena in multicomponent systems of gases. It appears in applications from various domains, 
for instance in dialysis, ion exchange, sedimentation and electrolysis. The problem of existence of solutions has been extensively studied in the literature, 
where local-in-time strong solutions \cite{Bo11,GiMa98,HeMePrWi17} and global-in-time weak solutions \cite{JuSt13} have been established. The uniqueness of strong solutions has been shown in \cite{HeMePrWi17,HuSa18}, however for weak solutions the problem has been open, with the only available result
being the weak-strong uniqueness property \cite{HuJuTz22}. We show here that weak solutions to the Maxwell-Stefan system are unique, provided the molar fluxes are essentially bounded.

 %----------------------------------------------------------------------------------------
 %----------------------------------------------------------------------------------------
 
  The Maxwell-Stefan system describes the evolution of the vector function $\cc=(c_1,\dots,c_n)$ where $c_i(x,t)$ is the concentration and $u_i (x,t)$ the velocity of the $i$-th component, $i=1,\dots,n$.
  Their dynamics is described by $n$ equations for the conservation of mass,
\begin{equation}
            \label{eq:mass} 
            \partial_t c_i + \nabla\cdot J_i  = 0 \, ,
\end{equation}
with $J_i = c_i u_i$ the molar fluxes and $u_i$ the diffusional velocities. System \eqref{eq:mass} is complemented by the constrained linear system
\begin{align}[left=\empheqlbrace]
            \label{eq:linear-system}
            -\sum_{{j=1},{j\not=i}}^n \frac{c_jJ_i - c_iJ_j}{D_{ij}} & = \nabla c_i  \quad \quad  i=1,\dots,n \, ,
            \\
            \label{eq:constraint}
            \sum_{i=1}^n J_i &= 0,
\end{align} 
for determining the molar fluxes $J_i(x,t)$, where $D_{ij}=D_{ji}>0$, $i\not=j$, are the diffusion coefficients modeling binary interactions between the components. 
We will consider the problem on  $\T^3 \times (0,T)$, where $\T^3$ is the $3$-dimensional torus and $T>0$. 
The system is completed with the initial conditions, 
	\begin{equation}
		\label{eq:inibd} c_i(x,0) = c_i^0(x) \quad  \text{ on } \T^3 \, .
	\end{equation}
The initial concentrations $c_i^0(x)\geq0$ are such that $\sum_{i=1}^nc_i^0(x)=1$ on $\T^3$, then  the constraint \eqref{eq:constraint} propagates
and implies $\sum_{i=1}^nc_i =1$, for $x \in \torusd$, $t >0$. The property $\sum_{i=1}^n c_i = 1$ provides a consistency condition for the constrained linear system \eqref{eq:linear-system}-\eqref{eq:constraint} which is uniquely solvable; 
see \cite{GiMa98,Bo11,JuSt13,HuJuTz19,HuJuTz22} for solvability properties of the linear system \eqref{eq:linear-system}-\eqref{eq:constraint}.

The system \eqref{eq:mass}-\eqref{eq:constraint} is endowed with an entropy functional, which for simple gases takes the form
 \begin{equation}     
        \label{eq:entropy}
        H(\cc)= \int_{\T^3} \sum_{i=1}^n c_i(\ln c_i-1)\dx \, .
    \end{equation}
The functional is well-defined even for $c_i=0$ and under natural conditions on the initial data leads
to a global existence theory \cite{JuSt13} for weak solutions of \eqref{eq:mass}-\eqref{eq:constraint},
by exploiting  the entropy structure in conjunction with the properties of the friction operator.

A commonly used method for stability proceeds by computing the evolution of the relative entropy, a method
applicable when one of the solutions is a strong solution. It provides  \cite{HuJuTz22} a weak-strong uniqueness result 
under the assumptions that the strong solution is sufficiently smooth and presents no anomalous diffusion (concentrations in the entropy identity).
The required smoothness of the strong solution is considerably more than what is provided by the global existence result in \cite{JuSt13}
and regularity results for weak solutions (even if expected) are currently not available. The strong solution in \cite{Bo11} is only locally defined in time
and  the weak-strong uniqueness result \cite{HuJuTz22} yields uniqueness so long as a strong solution exists.

To exploit the relative entropy under reduced regularity assumptions two new tools are introduced: (i) the relative entropy is replaced by the
symmetrized relative entropy 
\begin{equation}
		\label{eq:symrelen} H^{sym} (\cc,\bar{\cc}) := \int_{\T^3} \sum_{i=1}^n(\ln{ c_i }-\ln \bar{c}_i )(c_i-\bar{c}_i)\ \dx
\end{equation}
which satisfies a nice formal identity (see \eqref{eq:entropysystem}). (ii) For the rigorous derivation of the relative entropy identity we introduce
renormalized solutions for the Maxwell-Stefan system. (iii) Finally,  the singular behavior at the origin needs to be approximated, and this is achieved by introducing
a companion Maxwell-Stefan system for shifted solutions.

Renormalized solutions were introduced in the study of the Boltzmann equation \cite{DiLiBol89} and 
extended to transport equations \cite{DiLi89}, nonlinear elliptic problems \cite{Mu93},  equations of fluid mechanics \cite{Li96}, nonlinear parabolic problems \cite{BlMu97}. Closest to the use here are the references  \cite[pp. 227]{Li96}, \cite[Sec 4.1.5]{Feireisl03} concerning the compressible Navier-Stokes system. 
To motivate the definition of a renormalized solution, compute the evolution of a smooth function $\beta(c_i)$. Formal multiplication of \eqref{eq:mass} by $\beta'(c_i)$ gives: 	
\begin{equation} \label{eq:renorm}
		\partial_t \beta(c_i) + \nabla\cdot(\beta'(c_i) c_iu_i) - 2 c_i\beta''(c_i) \nabla \sqrt{c_i} \cdot \sqrt{c_i} u_i = 0
\end{equation} 
From the existence theory of weak solutions \cite{JuSt13}, we know that $\nabla \sqrt{c_i} \cdot \sqrt{c_i} u_i  \in L^1$ and use this information 
to give meaning to  \eqref{eq:renorm}.

The formal identity \eqref{eq:entropysystem} leads after integration to 
\begin{equation}
\label{eq:Gronwall}
\begin{aligned}
&\frac{\textnormal{d}}{\dt} \int_\torusd \Big ( \sum_i (\ln c_i  - \ln \bar{c}_i) ( c_i  - \bar{c}_i )\Big )  dx  \\
&\qquad + 
\int_\torusd  \underbrace{\sum_{i,j}  {\frac{1}{ 2 D_{i j}} ( c_i c_j  + \bar{c}_i \bar{c}_j ) \big | (u_i - \bar{u}_i ) - (u_j - \bar{u}_j) \big |^2 }}_{Q} dx
= E_1 + E_2 \\[5pt]
\end{aligned}
\end{equation}
where $E_1$, $E_2$ are error terms. 
Deriving the identity \eqref{eq:Gronwall} is the main difficulty; in the process we employ renormalized solutions
and an adaptation of  Krushkov's method of doubling of variables.
It is also proved that weak solutions of the regularity class provided by the existence theory \cite{JuSt13} yield renormalized solutions. 
As an application of the latter result, we obtain uniqueness of weak solutions. 
The obtained uniqueness theorem is conditional, in the sense that it holds for global weak solutions of additional regularity than that provided in \cite{JuSt13}; 
namely, it is valid for solutions of regularity $\nabla c_i \in L^\infty_{t,x}$ versus the regularity  
$\nabla \sqrt{c_i} \in L^2_{t,x}$ provided by the existence theory. On the other hand, it appears to exploit the full  capacity of the relative entropy identity, see the discussion in section \ref{sec:discussion}.

An outline of the contents follows: In section \ref{sec:comparison}, we state the formal estimates,
motivated by the transport equation structure, leading to the transport identity \eqref{eq:entropysystem} for the symmetrized relative entropy.
Renormalized solutions are used to give a meaning to \eqref{eq:Gronwall}: the approach, definitions, and 
obtained results are outlined in section \ref{sec:description}.
In section \ref{wk-ren}, we prove that weak solutions of the  Maxwell-Stefan system of the regularity class provided by the existence theory \cite{JuSt13} 
induce renormalized solutions. This hinges on two lemmas, inspired by the work  \cite{Fi15} on renormalized solutions of reaction-diffusion systems,
proved in section  \ref{wk-ren}. Thereafter, we focus on renormalized solutions of the Maxwell-Stefan system.
In section \ref{sec3}, we apply the method of doubling of variables to renormalized solutions in order to derive the symmetrized relative entropy identity
stated in Theorem \ref{theorem4}.
Then we have to resolve the following problem. The derived identity \eqref{eq:relenid} holds for test functions $\beta \in C^2([0,\infty))$ and this class does not include the
logarithm due to the singularity at the origin. To compensate, we introduce in section \ref{sec5} an approximation at the origin and the observation that a shifted solution
satisfies a companion Maxwell-Stefan system.
The last step,  performed in section \ref{sec6}, combines information from \cite{HuJuTz22} on the spectral gap of an operator associated to the frictional relative work $Q$
with a delicate estimation of the error terms $E_1$ and $E_2$ to circumvent the logarithmic singularities 
of the symmetrized relative entropy.
It leads to a uniqueness result for renormalized solutions of the Maxwell-Stefan system, see Theorem \ref{theorem-uniqueness}. Due to the equivalence of
weak and renormalized solutions, it provides Corollary \ref{thm-weak} on uniqueness of weak solutions.

%-----------------------------------------------------
%-----------------------------------------------------
%-----------------------------------------------------
%-----------------------------------------------------

\section{Formal comparison of solutions of transport equations}
\label{sec:comparison}

The entropy \eqref{eq:entropy} of the Maxwell-Stefan system  plays an important role in the existence theory of weak solutions \cite{JuSt13}, 
and the relative entropy 
\begin{equation}
 \label{eq:relen}
        H(\cc | \bcc)  = \int_\torusd  \sum_{i=1}^n  \Big (  c_i \ln \frac{c_i}{\bar{c}_i} - (c_i - \bar{c}_i) \Big )  \dx
 \end{equation}
was used in \cite{HuJuTz22} to estimate
 the distance between two solutions $\cc$ and $\bcc$. Here, we will employ instead a symmetrized version of the relative entropy
 \begin{equation}
 \label{eq:relentropy}
       H^{sym} ( \cc , \bcc) :=  H(\cc | \bcc) + H(\bcc | \cc) = \int_\torusd \sum_{i=1}^n (\ln c_i - \ln \bar{c}_i ) (c_i - \bar{c}_i)  \dx
 \end{equation}
The quantity $(\ln c_i - \ln \bar{c}_i ) (c_i - \bar{c}_i)$ is positive for $0 < c_i, \bar{c}_i < 1$ when $c_i \ne \bar{c}_i$. It takes the value $\infty$
when one out of $c_i$ or $\bar{c_i}$ vanishes, and it can be defined to be $+\infty$ when both vanish. Hence, it can serve to measure
the distance between two solutions.

For a renormalized solution corresponding to \eqref{eq:renorm}, an entropy functional is defined by
\begin{equation}\label{eq:renorm-entropy} 
   H_B (\cc)= \int_\torusd   \sum_{i=1}^n\int_0^{c_i}\beta(s)\ds \, \dx \, .
\end{equation} 
Setting $B(s):=\int_0^s\beta(\tau)\dtau$, the symmetrized relative entropy of two renormalized solutions $\cc$ and $\bar{\cc}$ is given by
 \begin{equation}\label{eq:renorm-relen}
            H^{sym}_B (\cc,\bar{\cc}) =
                \int_\torusd \sum_{i=1}^n (\beta(c_i)-\beta(\bar{c}_i))(c_i-\bar{c}_i) \, \dx \, .
 \end{equation} 
When $\beta$ is a monotone map, the symmetrized relative entropy is nonnegative and offers a suitable tool to measure the
distance between $\cc$ and $\bar{\cc}$.

 \subsection{Uniqueness for the heat equation}\label{unheat}
 In this section we present a uniqueness/stability estimate for the heat equation. 
 The heat equation is first written in the form suggested by the so-called Otto calculus \cite{Ot01}
 \begin{align}
 \label{heateq}
 \del_t \rho + \diver \rho u = 0  \qquad 
  u  = - \nabla \ln \rho.
\end{align}
 
 We present an identity based only on the transport equation, $\del_t \rho + \diver \rho u = 0 $,
 which yields  a relative entropy comparison between two solutions of \eqref{heateq}. 
 Let $(\rho, u)$ and $(\brho, \bu )$ be two solutions of the transport equation and we write down the identities
 \begin{align}
 \del_t \rho + \diver \rho u &= 0  
 \label{releq1}
 \\
 \del_t \ln \bar{\rho} + \diver \bar{u} + \nabla \ln \bar{ \rho} \cdot \bar{u} &= 0.
 \label{releq2}
 \end{align}
 We multiply \eqref{releq1} by $\ln \bar{\rho}$ and \eqref{releq2} by $\rho$ and add to obtain
\begin{equation}\label{eqn1}
\del_t (\rho \ln \bar{\rho} )+ \diver ( \rho  \ln \bar{\rho}  \, u + \rho \bar{u} )  + \nabla \ln \frac{\bar{\rho}}{\rho} \cdot \rho \bar{u} - \nabla \ln \bar{\rho} \cdot \rho u = 0.
 \end{equation}
 In particular, we have
 \begin{equation}\label{eqn2}
\del_t (\rho \ln \rho )+ \diver ( \rho  \ln  \rho  \, u + \rho u )   - \nabla \ln  \rho \cdot \rho u = 0.
 \end{equation}

 Next, we add  \eqref{eqn2} with the corresponding equation for $(\bar{\rho}, \bar{u})$, and then from the result we subtract \eqref{eqn1} and the analog
 of \eqref{eqn1} with $(\rho, u)$ and  $(\bar{\rho}, \bar{u})$ interchanged. After re-arranging the terms we obtain
 \begin{equation}
 \label{relensym}
 \begin{aligned}
 \del_t \big ( (\ln \rho - \ln \brho) (\rho - \brho) \big ) &+ \diver \big ( (\ln \rho - \ln \brho) (\rho u - \brho \bu) + (\rho - \brho) (u - \bu) \big )
 \\
 &-  (\rho + \brho) \nabla (\ln \rho - \ln \brho) \cdot  (u - \bu) = 0.
  \end{aligned}
 \end{equation}
We emphasize that up to here we only used the transport equation \eqref{releq1}.

Now, consider the heat equation \eqref{heateq} and introduce the  formula $u = - \nabla \ln \rho$ to \eqref{relensym}.
After an integration, we obtain the identity
\begin{equation}
\frac{d}{dt} \int_\Omega \big ( (\ln \rho - \ln \brho) (\rho - \brho) \big ) \dx + \int_\Omega  (\rho + \brho)  | \nabla (\ln \rho - \ln \brho) |^2 \dx = 0
\end{equation}
which provides an alternative approach to show uniqueness and stability for the heat equation, based on its entropy structure.

\subsection{The symmetrized relative entropy identity for the Maxwell-Stefan system}
\label{sec:symrelen}
Consider next the Maxwell-Stefan system and return to the notation $c_i$ for the densities and $J_i = c_iu_i$ for the fluxes.  Using
the transport equations  \eqref{eq:mass}  for the individual components and following the derivation of section \ref{unheat} we obtain the identity
\begin{equation}
\begin{split}
\del_t & \Big ( \sum_i (\ln c_i  - \ln \bar{c}_i) ( c_i  - \bar{c}_i )\Big )  \\
&+ \diver \Big ( \sum_i (\ln c_i  - \ln \bar{c}_i) ( c_i u_i  - \bar{c}_i \bar{u}_i) + ( c_i  - \bar{c}_i )(u_i - \bar{u}_i ) \Big )
\\
&\qquad -  \sum_i \Big ( \nabla (\ln c_i  - \ln \bar{c}_i)  \cdot (c_i + \bar{c}_i) (u_i - \bar{u}_i ) \Big ) = 0.
\end{split}
\label{eq:entropysystem1}
\end{equation}
Next, we use the algebraic system \eqref{eq:linear-system},
$$
 -\sum_{j=1}^n \frac{1}{D_{ij}} c_j (u_i - u_j)  = \nabla \ln c_i
 $$
and the symmetry of $D_{ij}$ to express
$$
\begin{aligned}
I &=  -  \sum_i   (c_i + \bar{c}_i) (u_i - \bar{u}_i) \cdot \nabla (\ln c_i  - \ln \bar{c}_i)  \cdot (c_i + \bar{c}_i) (u_i - \bar{u}_i )
\\
&=  \quad  \sum_{i,j} \tfrac{1}{2} c_i c_j \frac{1}{ D_{i j}} \big | (u_i - \bar{u}_i ) - (u_j - \bar{u}_j) \big |^2 + \sum_{i ,j} (c_j - \bar{c}_j) (u_i - \bar{u}_i) \cdot \frac{c_i}{D_{i j}} (\bar{u}_i - \bar{u}_j)
\\
&\quad + \sum_{i,j} \tfrac{1}{2} \bar{c}_i \bar{c}_j \frac{1}{ D_{i j}} \big | (u_i - \bar{u}_i ) - (u_j - \bar{u}_j) \big |^2 + \sum_{i ,j} (c_j - \bar{c}_j) (u_i - \bar{u}_i) \cdot \frac{\bar{c}_i}{D_{i j}} (u_i - u_j).
\end{aligned}
$$

Combining with \eqref{eq:entropysystem1} we arrive at the (formal) identity
\begin{equation}
\begin{split}
\del_t & \Big ( \sum_i (\ln c_i  - \ln \bar{c}_i) ( c_i  - \bar{c}_i )\Big )  \\
&\quad + \diver \Big ( \sum_i (\ln c_i  - \ln \bar{c}_i) ( c_i u_i  - \bar{c}_i \bar{u}_i) + ( c_i  - \bar{c}_i )(u_i - \bar{u}_i ) \Big )
\\
&\quad + \sum_{i,j}  \frac{1}{ 2 D_{i j}} ( c_i c_j  + \bar{c}_i \bar{c}_j ) \big | (u_i - \bar{u}_i ) - (u_j - \bar{u}_j) \big |^2 
\\
&= -  \sum_{i ,j} (c_j - \bar{c}_j) (u_i - \bar{u}_i) \cdot \frac{1}{D_{i j}}  \Big ( c_i (\bar{u}_i - \bar{u}_j) + \bar{c}_i (u_i - u_j) \Big ).
\end{split}
\label{eq:entropysystem}
\end{equation}
The identity \eqref{eq:entropysystem} will guide the analysis of uniqueness in the following sections. It bears analogies to the identity used 
for weak-strong uniqueness in \cite{HuJuTz22}, but here the evolution is computed for the symmetrized relative entropy.
It should be compared to Corollary \ref{corollary5} for $\delta = 0$. 
In section \ref{sec6.1} it is complemented with
spectral properties of the matrix in \eqref{eq:linear-system} that will quantify the effect of friction. 

%-----------------------------------------------------
%-----------------------------------------------------
%-----------------------------------------------------
%-----------------------------------------------------

\section{Description of results}\label{sec:description}
We work on the torus $\torusd$ with initial data that satisfy 
\begin{equation}
\label{assdata}
 0\leq c_i^0\in L^\infty({\T^3}) \, , \; \; i = 1, \dots, n  \, ; \qquad  \sum_{i=1}^nc_i^0=1 \, ;  \qquad  H(\cc^0)<+\infty \, .
 \tag{A}
\end{equation}
The solution propagates from the data the constraint  $\sum_i c_i = 1$. We recall the notion of weak solution.
 \begin{definition} \label{def-weak}
        A function $\cc=(c_1,\dots,c_n)$ is a weak solution of \eqref{eq:mass}-\eqref{eq:inibd}, if for $i\in\{1,\dots,n\}$: 
        \begin{itemize}
			\item[(i)]  $0\leq c_i\leq1$ and ${\displaystyle \sum_{i=1}^n c_i=1}$,
           \item [(ii)] $\nabla \sqrt{c_i}  \in L^2_\textnormal{loc}({\T^3}\times(0,\infty))  \, ,   \; \sqrt{c_i}u_i \in L^2_\textnormal{loc}({\T^3}\times(0,\infty))  )$,
            \item[(iii)] for any test function $\phi_i \in C_c^\infty({\T^3}\times[0,\infty))$, we have
		  \begin{equation} \label{eq:weak}
                \int_{\T^3}  c_i^0(x)\phi_i(x,0) \dx + \int_0^\infty \int_{\T^3} c_i \partial_t \phi_i \dx\dt+\int_0^\infty \int_{\T^3} c_iu_i \cdot \nabla \phi_i \dx \dt= 0
		  \end{equation} 
		 and  $c_iu_i$ satisfies  \eqref{eq:linear-system}-\eqref{eq:constraint}.  
           \item[(iv)]  ${\displaystyle \lim_{t \to 0} \int_\torusd | c_i (x, t) - c_i^0 (x) | dx \to 0. } $
        \end{itemize}   
 \end{definition} 
   
 \smallskip 
 \noindent
 Existence of weak solutions is established in \cite{JuSt13} for bounded domains under no-flux boundary conditions. Using  \cite{JuSt13} 
 and spectral properties for the Maxwell-Stefan system \eqref{eq:linear-system}-\eqref{eq:constraint} from \cite{HuJuTz22},  property (ii) implies the regularity 
\begin{equation}\label{regularwk}
c_i\in C^0_\textnormal{loc}([0,\infty);L^2({\T^3})), \quad \sqrt{c_i}u_i\in L^2_\textnormal{loc}({\T^3}\times(0,\infty)) 
\end{equation}
and gives meaning to the terms of \eqref{eq:weak}.

The definition of renormalized solutions is motivated by the formal identity \eqref{eq:renorm} and the regularity information provided by the existence theory,
notably that $\nabla \sqrt{c_i} \cdot \sqrt{c_i} u_i \in L^1_{\textnormal{loc}} ( \torusd \times (0, \infty) )$.
%	Let \[ \mathcal{B}=\{\beta\in C^1(\R) \textnormal{ s.t. } \sqrt{s}\beta'(s),s\beta''(s) \textnormal{ are bounded in } \mathbb{R}\}. \]
%	Then, we have the following definition:   
For test functions $\beta$ of class $C^2 ([0,\infty))$ we introduce the definition:
	
	\begin{definition}\label{def-renorm}
		A function $\cc = (c_1,\dots,c_n)$ is a renormalized solution of \eqref{eq:mass}-\eqref{eq:inibd}, if for $i\in\{1,\dots,n\}$ it satisfies:
\begin{itemize}
           \item[(i)]  $0\leq c_i\leq1$ and ${\displaystyle \sum_{i=1}^n c_i=1}$,
           \item [(ii)] $\nabla \sqrt{c_i}  \in L^2_\textnormal{loc}({\T^3}\times(0,\infty))  \, ,   \; \sqrt{c_i}u_i \in L^2_\textnormal{loc}({\T^3}\times(0,\infty))  )$, 
			\item[(iii)] for any test function $\phi_i \in C_c^\infty({\T^3}\times[0,\infty))$ and for $\beta \in C^2 ([0,\infty))$
			\begin{equation} \label{eq:weak-renorm}
				\begin{split}
					& \int_{\T^3}\beta(c_i^0)\phi_i(x,0)\dx+\int_0^\infty \int_{\T^3}  \beta(c_i) \partial_t \phi_i \dx \dt \\
                    & + \int_0^\infty \int_{\T^3} \sqrt{c_i}\beta'(c_i) \sqrt{c_i}u_i \cdot \nabla \phi_i \dx \dt + 2 \int_0^\infty\int_{\T^3} c_i \beta''(c_i) \nabla \sqrt{c_i} \cdot \sqrt{c_i}u_i \phi_i \dx \dt=0
				\end{split}
			\end{equation}
		and  $c_iu_i$ satisfies  \eqref{eq:linear-system}-\eqref{eq:constraint}.  
			\item[(iv)] ${\displaystyle \lim_{t \to 0} \int_\torusd | \beta ( c_i (x, t) )  -  \beta ( c_i^0 (x) )  | \dx \to 0 } $, for $\beta \in C^2([0,\infty))$.
		\end{itemize}
	\end{definition}
	
\smallskip
\noindent
Renormalized solutions are often introduced to handle the behavior near infinity. Here, the solution is bounded and the goal
is to handle the behavior at $c_i = 0$. Within the regularity class of weak solutions and for $\beta \in C^2([0,\infty))$ all terms in \eqref{eq:weak-renorm}
are well defined.

To this degree, we start by studying the relation between weak and renormalized solutions. We  show that the weak solutions constructed in \cite{JuSt13} have sufficient regularity to induce renormalized solutions in the sense of Definition \ref{def-renorm}:
    
    \begin{theorem} \label{thm-weakrenorm}
    	Let $\cc$ be a weak solution of \eqref{eq:mass}-\eqref{eq:linear-system} in the sense of Definition \ref{def-weak} of regularity \eqref{regularwk}.
	Then $\cc$ is a renormalized solution of \eqref{eq:mass}-\eqref{eq:linear-system}.
    \end{theorem}
   
Then, the problem of uniqueness of weak solutions reduces to showing uniqueness for renormalized solutions. The objective is to make the derivation of the symmetrized relative entropy identity \eqref{eq:entropysystem1} (equivalently \eqref{eq:entropysystem}) rigorous. At the same time \eqref{eq:symrelen} only allows for non-vanishing solutions because of the logarithmic term $\ln c_i$. 
Thus, we introduce a symmetrized relative entropy, which is regularized so as to also allow for zero concentrations:
	\begin{equation}
		\label{eq:symrelenapp} F(\cc,\bar{\cc}) := \int_{\T^3} \sum_{i=1}^n (\ln{(c_i + \delta)}-\ln(\bar{c}_i+\delta))(c_i-\bar{c}_i)\ \dx
	\end{equation}
where $\delta>0$ is a constant to be selected. The strategy will be to calculate a Gr\"onwall-type inequality for the functional $F(\cc,\bar{\cc}) > 0$.
    
The first step is to derive an identity for the symmetrized relative entropy between  two renormalized solutions, using the method of doubling the variables. 
This method, due to Kruzhkov \cite{Kr70}, was further developed in \cite{Ca94}, and allows to treat one solution as a constant with respect to the evolution of the other, 
in order to derive a relative entropy identity, overcoming the lack of regularity. Compared to the use of doubling of variables in
hyperbolic problems, we here employ also the regularity properties emerging from the existence theory.

	The work proceeds as follows: First we establish a symmetrized relative entropy inequality between two renormalized solutions: 	
	\begin{theorem}\label{theorem4}
		Let $\cc$ and $\bar{\cc}$ be renormalized solutions of \eqref{eq:mass}-\eqref{eq:linear-system} emanating from the same initial data. 
		Then for any $T>0$ and $\beta \in C^2([0,\infty))$, the following identity holds: 
		\begin{equation}\label{eq:relenid}
			\begin{split}
%				&F(\cc,\bar{\cc})(T) 
                 & \sum_{i=1}^n\int_{\T^3} \big ( \beta (c_i) - \beta (\bar{c}_i) \big ) ( c_i - \bar{c}_i) \, \dx \Big |_{t = T}
                 \\
				&\quad = \sum_{i=1}^n \int_0^T \int_{\T^3} \nabla \left(\beta(c_i) - \beta(\bar{c}_i)  \right)  \cdot (c_iu_i-\bar{c}_i\bar{u}_i) \dx \dt 
				\\ &\qquad + \sum_{i=1}^n \int_0^T\int_{\T^3} \Big( \nabla \big[(c_i - \bar{c}_i)\beta'(c_i)\big] \cdot c_i u_i  - \nabla \big[(c_i - \bar{c}_i)\beta'(\bar{c}_i)\big] \cdot \bar{c}_i \bar{u}_i \Big) \dx\dt.
			\end{split}
		\end{equation} 
	\end{theorem}
	
	Now, the choice of $\beta(s)=\ln(s+\delta)$, which is indeed in $C^2([0,\infty))$, in \eqref{eq:relenid} implies the following result:
	
	\begin{corollary}\label{corollary5}
		Let $\cc$ and $\bar{\cc}$ be renormalized solutions of \eqref{eq:mass}-\eqref{eq:linear-system}. Then, for any $T>0$ and $\delta>0$:
		\begin{eqnarray}\label{eq:id}
			&& F(\cc,\bar{\cc})(T)  \\ \nonumber &&= \sum_{i=1}^n\int_0^T \int_{\T^3} (c_i + \bar{c}_i+2 \delta) \ \nabla \big(\ln(c_i + \delta) - \ln(\bar{c}_i+\delta)\big) \cdot \left( \frac{c_i u_i}{c_i + \delta} - \frac{\bar{c}_i \bar{u}_i}{\bar{c}_i + \delta} \right) \dx \dt.
		\end{eqnarray} 
	\end{corollary}
	
	We remark that up to  Corollary \ref{corollary5} only the transport structure \eqref{eq:mass} was used. The question becomes how the
	formal identity \eqref{eq:entropysystem1} is affected by the introduction of the corrections in the logarithmic terms. 
To achieve this, a change of variables is quite helpful, namely, the functions  $d_i = c_i + \delta$ and $v_i = \frac{c_i u_i}{c_i + \delta}$ 
satisfy an approximate Maxwell-Stefan system in the form \eqref{eq:shifted main}. This motivates to express the error estimate for the new system.
A careful estimation of the terms involved in \eqref{eq:id} yields the uniqueness of nonnegative renormalized solutions:
	
	\begin{theorem}\label{theorem-uniqueness}
		Let $\cc$ and $\bar{\cc}$ be renormalized solutions of \eqref{eq:mass}-\eqref{eq:linear-system} emanating from the same initial data. Then, if $c_iu_i$ and $\bar c_i\bar u_i$ are in $L^\infty({\T^3}\times(0,\infty))$, there exists at most one nonnegative renormalized solution to the Maxwell-Stefan system \eqref{eq:mass}-\eqref{eq:linear-system}.
	\end{theorem}

    As a consequence, the uniqueness of weak solutions follows:
    
    \begin{corollary} \label{thm-weak}
		Let $\cc$ and $\bar{\cc}$ be weak solutions of \eqref{eq:mass}-\eqref{eq:linear-system} emanating from the same initial data and suppose that
		$c_iu_i$ and $\bar c_i\bar u_i$ are in $L^\infty({\T^3}\times(0,\infty))$. 
		Then, $\cc =\bar{\cc}$, i.e. weak solutions are unique.
	\end{corollary}
	\color{black}
 
	The paper is organized as follows: In section \ref{wk-ren} we show that weak and renormalized solutions coincide (cf. Theorem~\ref{thm-weakrenorm}). For the proof of Theorem \ref{thm-weakrenorm}, we make use of two lemmas for generic transport equations. In section \ref{sec3}, the method of doubling the variables is used to derive a relative entropy identity for renormalized solutions, thus showing Theorem \ref{theorem4}. In section \ref{sec5}, we perform the calculations that deal with vanishing solutions and lead to Corollary \ref{corollary5}, and we present the necessary preparations introducing the shifted solution. Then in section \ref{sec6} we exploit the frictional structure of the Maxwell-Stefan inversion matrix, in order to show Theorem \ref{theorem-uniqueness}. We conclude the paper with a brief discussion of the results and their implications in section \ref{sec:discussion}.

%-----------------------------------------------------
%-----------------------------------------------------
%-----------------------------------------------------
%-----------------------------------------------------

\section{Equivalence between weak and renormalized solutions}\label{wk-ren}

It is evident from Definition \ref{def-renorm} that weak solutions are a special case of renormalized solutions. 
In this section, we investigate the other direction (weak $\Rightarrow$ renormalized) and prove Theorem \ref{thm-weakrenorm}.

    We first prove a result for generic transport equations, which is inspired by \cite[Lemmas 4 and 5]{Fi15} and which is used in the proof of Theorem \ref{thm-weakrenorm}.

\begin{lemma} \label{lemma1}
    Let $u \in C^0([0,T];L^2(\T^3))$, with $u\in L^2(0,T;H^1(\T^3))$ and $u^0 \in L^2(\T^3)$, and $z \in L^2(0,T;(L^2(\T^3))^3)$, be such that for all $\psi \in C_c^\infty({\T^3} \times [0,T))$
    
    \begin{equation} \label{1}
        - \int_0^T \int_{\T^3} u \del_t\psi \, \dx\dt - \int_{\T^3} u^0 \psi(x,0) \dx = \int_0^T \int_{\T^3} z \cdot \nabla \psi \, \dx\dt.
    \end{equation}

    \noindent
    Let $\mathcal{B} = \{ \beta \in C^2(\R; \R) \textnormal{ s.t. } |\beta(w)| \leq C |w|^2, \; |\beta'(w)| \leq C |w|, \; |\beta''(w)| \leq C  \}$. Then, for $\beta \in \mathcal{B}$ and $\psi \in C_c^\infty({\T^3} \times [0,T))$, there holds
    
    \[ \begin{split}
        & - \int_0^T \int_{\T^3} \beta(u) \del_t\psi \, \dx\dt - \int_{\T^3} \beta(u^0) \psi(x,0) \dx \\ 
        & \phantom{xxxxxxxxxx} = \int_0^T \int_{\T^3} \beta'(u) z \cdot \nabla\psi \, \dx\dt + \int_0^T \int_{\T^3} \psi \beta''(u) \nabla u \cdot z \, \dx\dt.
    \end{split} \]
\end{lemma}

\begin{proof}
    Let $\rho_\sigma$ denote a standard symmetric mollifier with respect to space. We choose the test function $\rho_\sigma * \psi$ in \eqref{1} and integrate by parts (in space) on the right-hand side, to obtain for all $\psi \in C_c^\infty({\T^3} \times [0,T))$

    \begin{equation} \label{2}
        - \int_0^T \int_{\T^3} (\rho_\sigma * u) \del_t\psi \, \dx\dt - \int_{\T^3} (\rho_\sigma * u^0) \psi(x,0) \dx = - \int_0^T \int_{\T^3} \diver(\rho_\sigma * z) \psi \, \dx\dt.
    \end{equation}

    Define the functions $f^\sigma,g^\sigma:\T^3 \times [-T,T]\to\mathbb{R}$ by:
        \begin{equation}\label{ext1}
            f^\sigma(x,t) = \begin{cases}
                (\rho_\sigma * u)(x,t), & \quad t > 0 \\
                (\rho_\sigma * u^0)(x), & \quad t < 0
            \end{cases} 
        \end{equation} and
        \begin{equation}\label{ext2}
            g^\sigma(x, t) = \begin{cases}
                \diver(\rho_\sigma * z)(x,t), & \quad t > 0 \\
                0, & \quad t < 0
            \end{cases}
        \end{equation}
    \noindent
    and notice that $f^\sigma,g^\sigma \in L^2(-T,T;L^2(\T^3))$. Then, for all $\varphi \in C_c^\infty({\T^3} \times (-T,T))$, $f^\sigma$ and $g^\sigma$ satisfy, due to \eqref{2},
    
    \begin{equation} \label{3}
        - \int_{-T}^T \int_{\T^3} f^\sigma(x,t) \del_t\varphi \, \dx\dt = -\int_{-T}^T \int_{\T^3} g^\sigma(x,t) \varphi \, \dx\dt.
    \end{equation} 

    Let, now, $\rho_\tau$ be a standard symmetric mollifier with respect to time. Using $\rho_\tau * \varphi$ as a test function in \eqref{3}, we get for all $\varphi \in C^\infty_c(\T^3\times(-T+2\tau,T+2\tau))$
    
    \[ - \int_{-T}^T \int_{\T^3} (\rho_\tau * f^\sigma) \del_t\varphi \, \dx\dt = -\int_{-T}^T \int_{\T^3} (\rho_\tau * g^\sigma) \varphi \, \dx\dt, \] 
    
    \noindent
    which implies that $\rho_\tau * f^\sigma \in C^1((-T+\tau,T-\tau);L^1({\T^3}))$, with the representation
    
    \[ \frac{\textnormal{d}}{\dt}(\rho_\tau * f^\sigma) = -\rho_\tau * g^\sigma, \quad \textnormal{in the $L^1({\T^3})$ sense}. \]

    Now, for $\beta \in \mathcal{B}$, we have by the chain rule
    
    \begin{equation} \label{ren}
        - \int_{-T}^T \int_{\T^3} \beta(\rho_\tau * f^\sigma) \del_t\varphi \, \dx\dt = -\int_{-T}^T \int_{\T^3}  \beta'(\rho_\tau * f^\sigma) (\rho_\tau * g^\sigma) \varphi \, \dx\dt
    \end{equation} 

    \noindent
    for all $\varphi \in C_c^\infty({\T^3} \times (-T+2\tau,T-2\tau))$.
    
    Keeping $\sigma>0$ fixed, we let $\tau\to0$. For the term on the left-hand side, $\beta$ is continuous and $\rho_\tau * f^\sigma \to f^\sigma$ a.e., hence $\beta(\rho_\tau * f^\sigma)\to\beta(f^\sigma)$ a.e. and due to $f^\sigma$ being in  $L^2(-T,T;L^2(\T^3))$ and the growth of $\beta$, we have \begin{align*}
        & |\beta(\rho_\tau * f^\sigma)\del_t\varphi| \leq C |\rho_\tau * f^\sigma|^2 \|\del_t\varphi\|_\infty \quad \textnormal{and} \\
        & \rho_\tau * f^\sigma \to f^\sigma, \textnormal{ in $L^2(K)$ for any compact $K\subset \T^3 \times (-T,T)$}.
    \end{align*} 
    
    \noindent
    Similarly, for the right-hand side, since $\rho_\tau * g^\sigma \to g^\sigma$ a.e. and $\rho_\tau * f^\sigma \to f^\sigma$ a.e. and $\beta'$ is continuous, we have $\beta'(\rho_\tau * f^\sigma) (\rho_\tau * g^\sigma) \to \beta'( f^\sigma) g^\sigma$ a.e. and due to the growth of $\beta'$ and $f^\sigma, g^\sigma$ being in $L^2(-T,T;L^2(\T^3))$, we have \begin{align*}
        & |\beta'(\rho_\tau * f^\sigma) (\rho_\tau * g^\sigma) \varphi| \leq C|\rho_\tau * f^\sigma||\rho_\tau * g^\sigma|\|\varphi\|_\infty \quad \textnormal{and} \\
        & \rho_\tau * f^\sigma \to f^\sigma, \textnormal{ in $L^2(K)$ for any compact $K\subset \T^3 \times (-T,T)$}, \\
        & \rho_\tau * g^\sigma \to g^\sigma, \textnormal{ in $L^2(K)$ for any compact $K\subset \T^3 \times (-T,T)$}.
    \end{align*} 
    
    Letting $\tau \to 0$ we can apply the dominated convergence theorem to deduce that

    \[ - \int_{-T}^T \int_{\T^3} \beta(f^\sigma) \del_t\varphi \, \dx\dt = -\int_{-T}^T \int_{\T^3}  \beta'(f^\sigma) (g^\sigma) \varphi \, \dx\dt, \]

    \noindent
    for all $\varphi \in C_c^\infty({\T^3} \times (-T,T))$ and because of \eqref{ext1}-\eqref{ext2} we recover
    
    \[ - \int_0^T \int_{\T^3} \beta(\rho_\sigma * u) \del_t\varphi \, \dx\dt - \int_{\T^3} \beta(\rho_\sigma * u^0) \varphi(x,0) \, \dx = -\int_0^T \int_{\T^3} \beta'(\rho_\sigma * u) \diver(\rho_\sigma * z) \varphi \, \dx\dt. \]

    We, now, integrate by parts on the right-hand side

    \[ \begin{split}
        & - \int_0^T \int_{\T^3} \beta(\rho_\sigma * u) \del_t\varphi \, \dx\dt - \int_{\T^3} \beta(\rho_\sigma * u^0) \varphi(x,0) \, \dx \\
        & \phantom{xx} = \int_0^T \int_{\T^3} \beta''(\rho_\sigma * u) \nabla (\rho_\sigma * u) \cdot (\rho_\sigma * z) \varphi \, \dx\dt + \int_0^T\int_{\T^3}\beta'(\rho_\sigma * u) (\rho_\sigma * z) \cdot \nabla\varphi \, \dx\dt
    \end{split} \] and let $\sigma\to0$. For the terms on the left-hand side, we have $\rho_\sigma * u \to u$ a.e. and the continuity of $\beta$ implies that $\beta(\rho_\sigma * u) \to \beta(u)$ a.e.. Moreover, by $u\in L^2(0,T;L^2(\T^3))$ and the growth condition on $\beta$, we see that 
    
    \begin{align*}
        & |\beta(\rho_\sigma * u) \del_t\varphi| \leq C|\rho_\sigma * u|^2 \| \del_t\varphi\|_\infty, \textnormal{ where} \\
        & \rho_\sigma * u \to u \textnormal{ in $L^2(K)$, for any compact $K \subset \T^3 \times (0,T)$}.
    \end{align*}

    For the first term on the right-hand side, we have $\beta''(\rho_\sigma * u) \nabla (\rho_\sigma * u) \cdot (\rho_\sigma * z) \to \beta''(u) \nabla u \cdot z$ a.e. and use the growth condition $|\beta''(w)|\leq C$ and the fact that $z, \nabla u \in L^2(0,T;(L^2(\T^3))^3)$ to get 
    
    \begin{align*}
        & |\beta''(\rho_\sigma * u) \nabla (\rho_\sigma * u) \cdot (\rho_\sigma * z)\varphi| \leq C |\nabla (\rho_\sigma * u)||\rho_\sigma * z|\|\varphi\|_\infty, \textnormal{ where} \\
        &\nabla(\rho_\sigma * u) \to \nabla u \textnormal{ and } \rho_\sigma * z \to z \textnormal{ in $L^2(K)$ for any compact $K \subset \T^3\times(0,T)$}.
    \end{align*}
    
    Finally, for the last term on the right-hand side $\beta'(\rho_\sigma * u)(\rho_\sigma * z) \to \beta'(u) z$ a.e. and since $|\beta'(w)| \leq C|w|$, $u\in L^2(0,T;L^2(\T^3))$ and $z\in L^2(0,T;(L^2(\T^3))^3)$ we obtain 

    \begin{align*}
        & |\beta'(\rho_\sigma * u)(\rho_\sigma * z)\cdot \nabla\varphi| \leq C |\rho_\sigma * u||\rho_\sigma * z|\|\nabla \varphi\|_\infty, \textnormal{ where} \\
        &\rho_\sigma * u \to u \textnormal{ and } \rho_\sigma * z \to z \textnormal{ in $L^2(K)$ for any compact $K \subset \T^3\times(0,T)$}.
    \end{align*}
    
    By the dominated convergence theorem, taking the limit $\sigma\to0$, we conclude

    \[ \begin{split}
        & - \int_0^T \int_{\T^3} \beta(u) \del_t\varphi \, \dx\dt - \int_{\T^3} \beta(u^0) \varphi(x,0) \, \dx \\
        & \phantom{xxxxxxxx} = \int_0^T \int_{\T^3} \beta''(u) \nabla u \cdot z \varphi \, \dx\dt + \int_0^T\int_{\T^3}\beta'(u) z \cdot \nabla\varphi \, \dx\dt.
    \end{split} \]
\end{proof}

    If the solution satisfies the bounds $0 \leq u \leq M$, for some $M>0$, the growth conditions on $\beta$ and its derivatives are not needed:

    \begin{lemma} \label{weak-renorm}
        Let $u \in C([0,T];L^2(\T^3)) \cap L^2(0, T; H^1(\T^3))$, with $0 \leq u(x,t) \leq M$, for some $M > 0$, $u^0 \in L^2(\T^3)$ and $z \in L^2(0, T; (L^2(\T^3))^3)$ satisfy for all $\varphi \in C_c^\infty(\T^3 \times [0, T))$

        \[ - \int_0^T \int_{\T^3} u \del_t\varphi \, \dx\dt - \int_{\T^3} u^0 \varphi(x,0) \dx = \int_0^T \int_{\T^3} z \cdot \nabla\varphi \, \dx\dt. \]

        \noindent
        Let $\beta \in C^2([0,\infty))$. Then, for all $\varphi \in C_c^\infty(\T^3 \times [0, T))$, there holds

        \[ \begin{split}
        & - \int_0^T \int_{\T^3} \beta(u) \del_t\varphi \, \dx\dt - \int_{\T^3} \beta(u^0) \varphi(x,0) \dx \\ 
        & \phantom{xxxxxxxxxx} = \int_0^T \int_{\T^3} \beta'(u) z \cdot \nabla\varphi \, \dx\dt + \int_0^T \int_{\T^3} \beta''(u) \nabla u \cdot z \varphi \, \dx\dt.
    \end{split} \]
    \end{lemma}

    \begin{proof}
        As in the proof of Lemma \ref{lemma1}, we derive \eqref{ren} and we pass to the limit $\tau \to 0$. 
        The limit is done via a similar (but simpler) argument using the bound $0 \leq u \leq M$ and the continuity of $\beta,\beta'$ and $\beta''$. In particular, since $u$ and $u^0$ take values between $0$ and $M$, the same holds for $\rho_\tau * f^\sigma$, and the continuity of $\beta$ and $\beta'$ implies that $\beta(\rho_\tau * f^\sigma)$ and $\beta'(\rho_\tau * f^\sigma)$ are bounded. This is enough to apply the dominated convergence theorem  to \eqref{ren} and pass to the limit $\tau\to0$. We, then, proceed as in the proof of Lemma \ref{lemma1} and use the same idea for the limit $\sigma\to0$.
    \end{proof}

\noindent
{\it Proof of Theorem \ref{thm-weakrenorm}.}
Let $\cc$ a weak solution satisfying (i)-(iv) of Definition \ref{def-weak}. For $\cc$ to be a renormalized solution, it has to satisfy (i)-(iv) of Definition \ref{def-renorm}. Conditions (i), (ii) are clear and imply $c_i \in L^2(0,T;H^1(\T^3))$ and $c_iu_i \in L^2(0,T;(L^2(\T^3))^3)$. Moreover, $c_i \in C([0,T];L^2(\T^3))$ (see \eqref{regularwk}). This allows for a direct application of Lemma \ref{weak-renorm} with $u=c_i$ and $z=c_iu_i$, which implies that $\cc$ satisfies \eqref{eq:weak-renorm} (notice that $2c_i\nabla\sqrt{c_i}\cdot\sqrt{c_i}u_i = \nabla c_i \cdot c_iu_i$). Finally, since $\beta\in C^2$, we have $\beta$ is Lipschitz continuous, hence (iv) of Definition \ref{def-renorm} follows due to condition (iv) of Definition \ref{def-weak}. Hence,  $\cc$ is a renormalized solution.
\hfill $\square$ 
\break

%-----------------------------------------------------
%-----------------------------------------------------
%-----------------------------------------------------
%-----------------------------------------------------

\section{Derivation of a symmetrized relative entropy identity}\label{sec3}

Let $c_i,\bar c_i$ be two renormalized solutions and we proceed to establish Theorem \ref{theorem4}. By the weak formulation for the difference $A(x,t):=c_i(x,t)-\bar{c}_i(x,t)$, corresponding to $\beta=\textnormal{id}$, we have  \begin{equation}\label{wk1}
    \int_\torusd A(x,0)\phi(x,0)\dx + \int_0^\infty \int_\torusd [A(x,t)\del_t\phi(x,t)+C(x,t)\cdot\nabla_x\phi(x,t)]\dx\dt = 0,
\end{equation} where $C(x,t):=c_iu_i(x,t)-\bar c_i\bar u_i(x,t)$ and $\phi (x,t) \in C^1_c (Q_\infty)$ a test function compactly supported in
$Q_\infty =  \torus^3 \times [0,\infty)$.
Likewise, write the weak formulation for the difference $B(y,\tau):=\beta(c_i(y,\tau))-\beta(\bar{c}_i(y,\tau))$, for a generic $\beta$, namely \begin{equation}\label{wk2}
    \begin{split}
        \int_\torusd B(y,0)\phi(y,0)\dy & + \int_0^\infty\int_\torusd B(y,\tau)\del_\tau\phi(y,\tau)+D(y,\tau)\cdot\nabla_y\phi(y,\tau)]\dy\dtau \\
        & \phantom{xxx} + \int_0^\infty \int_\torusd E(y,\tau)\phi(y,\tau)\dy\dtau = 0,
    \end{split}
\end{equation} where $\phi(y, \tau) \in C^1_c(Q_\infty)$,
$$
\begin{aligned}
D(y,\tau) &= \beta'(c_i(y,\tau))c_iu_i(y,\tau)-\beta'(\bar c_i(y,\tau))\bar c_i\bar u_i(y,\tau)
\\[5pt]
E(y,\tau) &= 2\sqrt{c_i}\beta''(c_i)\nabla\sqrt{c_i}\cdot c_iu_i-2\sqrt{\bar c_i}\beta''(\bar c_i)\nabla\sqrt{\bar c_i}\cdot \bar c_i \bar u_i. 
\end{aligned}
$$

Let $\varphi (z,s) \in C^1_c (Q_\infty)$, that is $\phi$ is periodic in space and  compactly supported on $\torusd\times [0,\infty)$.
Let $\rho(\sigma) \ge 0$ be a symmetric mollifier, supported in the ball centered at $0$ with radius $1$,  with $\int \rho(\sigma)  d\sigma = 1$.
Define 
\begin{equation} \begin{split}
    {\Phi}(x,t ; y , \tau) & = \varphi \left(\frac{x+y}{2},\frac{t+\tau}{2} \right) \;  \frac{1}{\epsilon}\rho\left(\frac{t-\tau}{\epsilon}\right)\prod_{\alpha=1}^3\frac{1}{\epsilon}\rho\left(\frac{x_\alpha-y_\alpha}{\epsilon}\right) 
    \label{testfunction}
    \\
    & =: \varphi \left(\frac{x+y}{2},\frac{t+\tau}{2} \right) \; \phi^\epsilon(x,y,t,\tau)
\end{split} 
\end{equation}
and note that $\Phi \in C^1_c ( Q_\infty \times Q_\infty)$ and if $z$ and $s$ denote the first and second argument of $\varphi$ respectively, we have 
\begin{align}
\label{id1}
    (\nabla_x+\nabla_y) {\Phi}=(\nabla_z  \varphi) \Big (\frac{x+y}{2},\frac{t+\tau}{2} \Big)  \;     \phi^\epsilon \\
\label{id2}
    (\del_t+\del_\tau) {\Phi}= (\del_s\varphi )\Big (\frac{x+y}{2},\frac{t+\tau}{2} \Big)  \; \phi^\epsilon.
\end{align}
$\Phi$ is jointly periodic in space: if $p\in\R^3$ is a period then $\Phi(x+p,t;y+p,\tau)  = \Phi(x,t;y,\tau)$.

Now, in \eqref{wk1} we employ the test function $\phi_B(x,t) = \int_0^\infty \int_{\T^3}B(y,\tau)\Phi(x,t;y,\tau)\dy\dtau$ 
 which is in $C_c^1(Q_\infty)$, since 
\begin{align*}
    \phi_B(x+p,t) & = \int_0^\infty \int_{\T^3}B(y,\tau)\Phi(x+p,t;y,\tau)\dy\dtau \\
    & = \int_0^\infty \int_{\T^3-p}B(y+p,\tau)\Phi(x+p,t;y+p,\tau)\dy\dtau \\
    & = \int_0^\infty \int_{\T^3}B(y,\tau)\Phi(x,t;y,\tau)\dy\dtau \\
    & = \phi_B(x,t). 
\end{align*}
The resulting
identity is integrated over $\torusd \times [0,\infty)$ and gives, using \eqref{id1},
 \[ 
 \begin{split}
    &\int_0^\infty \int_\torusd \int_0^\infty \int_\torusd A(x,t)B(y,\tau) (\del_t \Phi )
    + C(x,t)B(y,\tau) \cdot \Big( \nabla_z\varphi\big(\tfrac{x+y}{2},\tfrac{t+\tau}{2}\big)\phi^\epsilon-\nabla_y{\Phi} \Big )   \dx\dt\dy\dtau
    \\
    &\qquad =  - \int_0^\infty \int_\torusd \int_\torusd A(x,0)B(y,\tau)\Phi (x,0; y, \tau) \dx\dy\dtau  \, .\\
\end{split} \] 
Since $B(y,\tau)=\beta(c_i(y,\tau))-\beta(\bar c_i(y,\tau))\in H^1(\torusd)$ for $\tau$ fixed, we can integrate by parts in $y$ and use the periodic boundary conditions
to get 
\[ \begin{split}
    & \int_0^\infty\int_\torusd\int_\torusd A(x,0)B(y,\tau)\varphi\big(\tfrac{x+y}{2},\tfrac{\tau}{2}\big)\phi^\epsilon(x,y,0,\tau)\dx\dy\dtau \\
    & = - \int_0^\infty\int_\torusd\int_0^\infty\int_\torusd A(x,t)B(y,\tau) \del_t \big ( \Phi (x,t ; y, \tau ) \big ) \dx\dt\dy\dtau \\
    & - \int_0^\infty\int_\torusd\int_0^\infty\int_\torusd \left[C(x,t)B(y,\tau)\cdot (\nabla_z\varphi )\big(\tfrac{x+y}{2},\tfrac{t+\tau}{2}\big)\phi^\epsilon+C(x,t)\cdot\nabla_yB(y,\tau){\Phi}\right]\dx\dt\dy\dtau.
\end{split} \]

Similarly, we choose the test function $\phi_A(y,\tau) = \int_0^\infty\int_\torusd A(x,t)\Phi(x,t;y,\tau)\dx\dt$ in \eqref{wk2}, which is $C_c^1(Q_\infty)$ 
and follow similar steps to arrive at 
\[ \begin{split}
    & \int_0^\infty\int_\torusd \int_\torusd A(x,t)B(y,0)  \varphi\big(\tfrac{x+y}{2},\tfrac{\tau}{2}\big)\phi^\epsilon(x,y,t,0)  \dx\dy\dt \\
    & = - \int_0^\infty\int_\torusd \int_0^\infty\int_\torusd A(x,t)B(y,\tau)\del_\tau  \big ( \Phi (x,t ; y, \tau ) \big )   +  A(x,t)E(y,\tau) {\Phi} \; \dx\dt\dy\dtau \\
    & - \int_0^\infty\int_\torusd \int_0^\infty\int_\torusd \left[A(x,t)D(y,\tau)\cdot(\nabla_z\varphi )\big(\tfrac{x+y}{2},\tfrac{t+\tau}{2}\big)\phi^\epsilon +D(y,\tau)\cdot\nabla_xA(x,t) {\Phi}\right]\dx\dt\dy\dtau.
\end{split} \] 
Adding the two equations and using \eqref{id2}, we get 
\begin{align}
& \int_0^\infty\int_\torusd\int_\torusd A(x,0)B(y,\tau)\varphi\big(\tfrac{x+y}{2},\tfrac{\tau}{2}\big)\phi^\epsilon(x,y,0,\tau)\dx\dy\dtau  
\nonumber
\\
& + \int_0^\infty\int_\torusd \int_\torusd A(x,t)B(y,0)  \varphi\big(\tfrac{x+y}{2},\tfrac{\tau}{2}\big)\phi^\epsilon(x,y,t,0)  \dx\dy\dt 
\nonumber
\\
    & = - \int_0^\infty\int_\torusd \int_0^\infty\int_\torusd  A(x,t)B(y,\tau) \varphi_s\big(\frac{x+y}{2},\frac{t+\tau}{2}\big)\phi^\epsilon\dx\dt\dy\dtau 
 \nonumber
 \\
    & - \int_0^\infty\int_\torusd \int_0^\infty\int_\torusd \left[C(x,t)B(y,\tau)+A(x,t)D(y,\tau)\right]\cdot\nabla_z\varphi \big(\frac{x+y}{2},\frac{t+\tau}{2}\big)\phi^\epsilon\dx\dt\dy\dtau 
 \nonumber
 \\
    & - \int_0^\infty\int_\torusd \int_0^\infty\int_\torusd \left[C(x,t)\cdot\nabla_yB(y,\tau)+D(y,\tau)\cdot\nabla_xA(x,t)\right]\varphi\big(\frac{x+y}{2},\frac{t+\tau}{2}\big)\phi^\epsilon\dx\dt\dy\dtau 
    \nonumber
 \\
    & - \int_0^\infty\int_\torusd \int_0^\infty\int_\torusd A(x,t)E(y,\tau)\varphi\big(\frac{x+y}{2},\frac{t+\tau}{2}\big)\phi^\epsilon\dx\dt\dy\dtau.
\nonumber
\end{align}

We, turn to passing to the limit $\epsilon\to0$. This is effected by standard approximation results that are listed for the reader's convenience.
Lemma \ref{initialdata} indicates the subtle role played by the initial trace.

\begin{lemma}\label{young}
    Let $f(x,t)\in L^1(\R^d\times(0,\infty))$. Then \[ \begin{split}
        & \int_0^\infty\int_{\R^d}\int_0^\infty\int_{\R^d} f(y,\tau)\varphi\left(\frac{x+y}{2},\frac{t+\tau}{2}\right)\frac{1}{\epsilon}\rho\left(\frac{t-\tau}{\epsilon}\right)\frac{1}{\epsilon^d}\rho\left(\frac{x-y}{\epsilon}\right)\dy\dtau\dx\dt \\
        & \phantom{xxxxxxx} \to \int_0^\infty\int_{\R^d}f(x,t)\varphi(x,t)\dx\dt \quad \forall\varphi\in C_c(\R^d\times[0,\infty)),
    \end{split} \] where we denote \[ \rho\left(\frac{x-y}{\epsilon}\right) = \prod_{\alpha=1}^d\rho\left(\frac{x_\alpha-y_\alpha}{\epsilon}\right). \]
\end{lemma}

\begin{proof}
    Extend $f$ by setting \[ \bar{f}(x,t) = \begin{cases}
        f(x,t) & t>0 \\
        0 & \textnormal{else}
    \end{cases} \] and set $X = (x,t)$ and $Y = (y,\tau)$. Then, it is enough to show that for all $\varphi\in C_c(\R^d\times[0,\infty))$
    
    \[ \iint\bar{f}(Y)\varphi\left(\frac{X+Y}{2}\right)\frac{1}{\epsilon^{d+1}}\rho\left(\frac{X-Y}{\epsilon}\right)\dY\dX \to \int\bar{f}(X)\varphi(X)\dX. \]

    Indeed, using the change of variables $Y = X+\epsilon Z$, 
    \[ \begin{split}
        I & = \left|\iint \bar{f}(Y)\varphi\left(\frac{X+Y}{2}\right)\frac{1}{\epsilon^{d+1}}\rho\left(\frac{X-Y}{\epsilon}\right)\dY\dX - \int\bar{f}(X)\varphi(X)\dX\right| \\
        & = \left| \iint \left(\bar{f}(X+\epsilon Z)\varphi\left(X+\frac{1}{2}\epsilon Y\right) - \bar{f}(X)\varphi(X)\right)\rho\left(Z\right)\dZ\dX \right| \\
        & \leq \iint |\bar{f}(X+\epsilon Z)-\bar{f}(X)|\left|\varphi\left(X+\frac{1}{2}\epsilon Z\right)\right|\rho(Z)\dZ\dX \\
        & \phantom{xxxx} + \iint \left|\varphi\left(X+\frac{1}{2}\epsilon Z\right)-\varphi(X)\right||\bar{f}(X)|\rho(Z)\dZ\dX,
    \end{split} \]
where both terms tend to zero by the dominated convergence theorem since $f \in L^1$.
\end{proof}

\begin{lemma}\label{initialdata}
    Let $f(\cdot,t)\in L^1(\R^d)$ for all $t>0$ and $\lim_{t\to0}\int_{\R^d}|f(x,t)-f(x,0)|\dx = 0$. Then \[ \begin{split}
        & \int_{\R^d}\int_0^\infty\int_{\R^d} f(y,\tau)\varphi\left(\frac{x+y}{2},\frac{\tau}{2}\right)\frac{1}{\epsilon}\rho\left(-\frac{\tau}{\epsilon}\right)\frac{1}{\epsilon^d}\rho\left(\frac{x-y}{\epsilon}\right)\dy\dtau\dx \\
        & \phantom{xxxxxxx} \to \frac{1}{2}\int_{\R^d}f(x,0)\varphi(x,0)\dx \quad \forall\varphi\in C_c(\R^d\times[0,\infty)).
    \end{split} \] 
\end{lemma}

\begin{proof} We compare the difference between $K^\eps$ and $K$
\[ \begin{split}
        K^\epsilon &:= \int_{\R^d}\int_0^\infty\int_{\R^d} f(y,\tau)\varphi\left(\frac{x+y}{2},\frac{\tau}{2}\right)\frac{1}{\epsilon}\rho\left(-\frac{\tau}{\epsilon}\right)\frac{1}{\epsilon^d}  
             \rho\left(\frac{x-y}{\epsilon}\right)\dy\dtau\dx 
\\
        K &:= \frac{1}{2}\int_{\R^d}f(x,0)\varphi(x,0)\dx
    \end{split} \] 
By the change of variables $\frac{y-x}{\epsilon} = z$ and $\frac{\tau}{\epsilon} = \sigma$, the symmetry of $\rho $, and $\int_0^1 \rho (\sigma) d\sigma = 1/2$ we have 
 \[ \begin{split}
        |K^\epsilon - K| & \leq \int_{\R^d}\int_0^1\int_{|z|\leq1} \left| f(x+\epsilon z,\epsilon\sigma) \varphi\left(x+\frac{1}{2}\epsilon z, \frac{1}{2}\epsilon\sigma\right) - f(x,0)\varphi(x,0) \right| \rho(\sigma)\rho(z)\dz\dsigma\dx \\
        & \leq \int_{\R^d}\int_0^1\int_{|z|\leq1} \left| f(x+\epsilon z,\epsilon\sigma) \right| \left| \varphi\left(x+\frac{1}{2}\epsilon z, \frac{1}{2}\epsilon\sigma\right) - \varphi(x,0) \right| \rho(\sigma)\rho(z)\dz\dsigma\dx \\
        & \phantom{xxx} + \int_{\R^d}\int_0^1\int_{|z|\leq1} \left| f(x+\epsilon z,\epsilon\sigma) - f(x,0) \right| \left|\varphi(x,0) \right| \rho(\sigma)\rho(z)\dz\dsigma\dx =: J_1 + J_2,
    \end{split} \]
    where \[ \begin{split}
        J_1 & = \int_0^1 \rho(\sigma) \int_{|z|\leq1}\rho(z) \int_{\R^d} \left| f(x+\epsilon z,\epsilon\sigma) \right| \left| \varphi\left(x+\frac{1}{2}\epsilon z, \frac{1}{2}\epsilon\sigma\right) - \varphi(x,0) \right| \dx\dz\dsigma \\
        & \leq \mathcal{O}(\epsilon)\sup_{(z,s)}|\nabla_{(z,s)}\varphi(z,s)| \sup_{|\tau|\leq \epsilon} \int_{\R^d} \left| f(y,\tau) \right| \dy 
    \end{split} \]
    which goes to zero as $\epsilon\to0$, provided $\limsup_{\tau\to0}\|f(\cdot,\tau)\|_{L^1(\R^d)} \leq C$, where $C$ is a constant independent of $\epsilon$.

    For the second term, we have \[ \begin{split}
        J_2 & = \int_0^1\rho(\sigma)\int_{|z|\leq1}\rho(z)\int_{\R^d}\left| f(x+\epsilon z,\epsilon\sigma) - f(x,0) \right| \left|\varphi(x,0) \right| \dx\dz\dsigma \\
        & \leq \int_0^1\rho(\sigma)\int_{|z|\leq1}\rho(z)\int_{\R^d}\left| f(x+\epsilon z,\epsilon\sigma) - f(x+\epsilon z,0) \right| \left|\varphi(x,0) \right| \dx\dz\dsigma \\
        & \phantom{xxx} + \int_0^1\rho(\sigma)\int_{|z|\leq1}\rho(z)\int_{\R^d}\left| f(x+\epsilon z,0) - f(x,0)\right| \left|\varphi(x,0) \right| \dx\dz\dsigma \\
        & \leq \sup_x|\varphi(x,0)|\sup_{0\leq\sigma\leq1}\int_{\R^d}\left| f(y,\epsilon\sigma) - f(y,0) \right|\dy\dsigma \\
        & \phantom{xxx} + \frac{1}{2}\sup_x|\varphi(x,0)|\|f(\cdot+\epsilon z,0)-f(\cdot,0)\|_{L^1(\R^d)}
    \end{split} \]
Since $\lim_{\tau\to0}\int_{\R^d}|f(y,\tau)-f(y,0)|\dy = 0$ and $\lim_{z\to0}\|f(\cdot+z,0)-f(\cdot,0)\|_{L^1(\R^d)} = 0$ the last term goes to zero as $\epsilon\to0$. 
\end{proof}

Using the lemmas we take $\epsilon\to0$ to obtain
\[ \begin{split}
    & -\int_{\T^3} A(x,0)B(x,0)\varphi(x,0)\dx = \int_0^\infty\int_{\T^3} A(x,t)B(x,t)\varphi_t(x,t)\dx\dt \\
    & + \int_0^\infty\int_{\T^3}\left[C(x,t)B(x,t)+A(x,t)D(x,t)\right]\cdot\nabla_x\varphi(x,t)\dx\dt \\
    & + \int_0^\infty\int_{\T^3}\left[C(x,t)\cdot\nabla_xB(x,t)+D(x,t)\cdot\nabla_xA(x,t)\right]\varphi(x,t)\dx\dt \\
    & + \int_0^\infty\int_{\T^3} A(x,t)E(x,t)\varphi(x,t)\dx\dt.
\end{split} \]
Now, we choose \begin{equation}\label{local}
    \varphi(x,t)=\begin{cases}
    1 & t\in[0,T] \\
    \frac{T-t}{\sigma}+1 & T<t\leq T+\sigma \\
    0 & t>T+\sigma
\end{cases}
\end{equation} for some $\sigma>0$ and let $\sigma\to0$: \[ \begin{split}
    & \int_{\T^3} A(x,T)B(x,T)\dx-\int_{\T^3} A(x,0)B(x,0)\dx \\
    & = \int_0^T\int_{\T^3}\left[C(x,t)\cdot\nabla_xB(x,t)+D(x,t)\cdot\nabla_xA(x,t)+A(x,t)E(x,t)\right]\dx\dt.
\end{split} \]
Finally, substituting back $A,B,C,D, E$, summing the resulting equation over $i\in\{1,\dots,n\}$, 
and using the same initial conditions for the two solutions, we arrive at \eqref{eq:relenid}.

%-----------------------------------------------------
%-----------------------------------------------------
%-----------------------------------------------------
%-----------------------------------------------------

 \section{Approximation at the origin -- The shifted solution $(d_i , v_i )$. }\label{sec5}

An additional difficulty when employing the renormalized solutions occurs because of the singularity at the origin.
This will be addressed by introducing the shifted solution $(d_i , v_i)$. In preparation, we introduce a special function $\beta (c_i)$,
designed to approximate the singularity at the origin, and prove :

\bigskip
\noindent
{\it Proof of Corollary \ref{corollary5}.} Let $\beta(s) = \ln(s+\delta)$, for some fixed $\delta\in(0,1)$. Note that $\beta\in C^2([0,\infty))$, since for all $s \in [0,1]$: 
        $$ |\beta(s)|=|\ln(s+\delta)|\leq \max\{|\ln\delta|,\ln(1+\delta)\}$$ 
        $$|\beta'(s)| = \left|\frac{1}{s+\delta}\right| \leq \frac{1}{\delta}, $$ 
        $$|\sqrt{s} \beta''(s)| = \left|\frac{\sqrt{s}}{(s+\delta)^2}\right| \le \frac{1}{\delta^2}.$$ Hence, we can choose it in \eqref{eq:relenid} and carry out the following calculation:
		\begin{eqnarray*}
			&& \nabla \big[(c_i - \bar{c}_i)\beta'(c_i)\big] \cdot c_i u_i - \nabla \big[ (c_i - \bar{c}_i) \beta'(\bar{c}_i) \big] \cdot \bar{c}_i \bar{u}_i \\ &=& \nabla \left(1 - \frac{\bar{c}_i+\delta}{c_i + \delta}\right) \cdot c_i u_i - \nabla \left(\frac{c_i+\delta}{\bar{c}_i+\delta}-1\right) \cdot \bar{c}_i \bar{u}_i \\ &=& -\nabla \left(\frac{\bar{c}_i+\delta}{c_i + \delta}\right) \cdot c_i u_i - \nabla \left(\frac{c_i+\delta}{\bar{c}_i+\delta}\right) \cdot \bar{c}_i \bar{u}_i \\ &=& -\nabla \ln\left(\frac{\bar{c}_i+\delta}{c_i + \delta}\right) \cdot c_i u_i \frac{\bar{c}_i+\delta}{c_i + \delta} - \nabla \ln\left(\frac{c_i+\delta}{\bar{c}_i+\delta}\right) \cdot \bar{c}_i \bar{u}_i\frac{c_i+\delta}{\bar{c}_i+\delta} \\ &=& \nabla \big(\ln(c_i + \delta) - \ln(\bar{c}_i+\delta)\big) \cdot \left(c_i u_i \frac{\bar{c}_i+\delta}{c_i + \delta} - \bar{c}_i \bar{u}_i\frac{c_i+\delta}{\bar{c}_i+\delta}  \right)
		\end{eqnarray*}
		Then,
		\begin{eqnarray*}
			&&\nabla \big(\ln(c_i + \delta) - \ln(\bar{c}_i+\delta)\big) \cdot \left(c_i u_i \frac{\bar{c}_i+\delta}{c_i + \delta} - \bar{c}_i \bar{u}_i\frac{c_i+\delta}{\bar{c}_i+\delta}  \right) \\ && + \nabla \big(\ln(c_i + \delta) - \ln(\bar{c}_i+\delta)\big) \cdot (c_i u_i - \bar{c}_i \bar{u}_i) \\ &=& \nabla \big(\ln(c_i + \delta) - \ln(\bar{c}_i+\delta)\big) \cdot \left(c_i u_i \left(1+\frac{\bar{c}_i+\delta}{c_i + \delta}\right) - \bar{c}_i \bar{u}_i\left(1+\frac{c_i+\delta}{\bar{c}_i+\delta}\right) \right) \\ &=& (c_i + \delta + \bar{c}_i+\delta) \ \nabla \big(\ln(c_i + \delta) - \ln(\bar{c}_i+\delta)\big) \cdot \left( \frac{c_i u_i}{c_i + \delta} - \frac{\bar{c}_i \bar{u}_i}{\bar{c}_i + \delta} \right)
		\end{eqnarray*}
		Since the above calculation is purely algebraic, one would expect that it holds rigorously provided all the terms are well-defined.  Indeed, 
		$$ c_i \in L^\infty({\T^3} \times (0,T)), $$
		$$ \nabla \ln(c_i + \delta) = \frac{2\sqrt{c_i}}{c_i + \delta} \nabla \sqrt{c_i} \in L^2({\T^3} \times (0,T)),$$ 
		$$ \frac{c_i u_i}{c_i + \delta} = \frac{\sqrt{c_i}}{c_i + \delta} \ \sqrt{c_i}u_i \in L^2({\T^3} \times (0,T)) $$
		and thus \eqref{eq:id} is justified. \hfill $\square$ \break

Next, we introduce a new set of variables $d_i(x,t)$ and $v_i(x,t)$ :
 	\begin{eqnarray*}
 		d_i(x,t) &=& c_i(x,t) + \delta \\ 
 		v_i(x,t) &=& \frac{c_iu_i}{c_i + \delta}
 	\end{eqnarray*} for some $\delta\in(0,1)$.
Note that the flux is preserved in the sense that $c_i u_i = d_i v_i$, while the new flux also satisfies the constraint $\sum_{i=1}^n d_i v_i = 0$.
Morepver,  if $c_iu_i\in L^\infty({\T^3} \times (0,T))$, then $d_i v_i \in L^\infty({\T^3} \times (0,T))$ as well. Therefore, \eqref{eq:mass}-\eqref{eq:linear-system} reads: 
 	\begin{equation}
 		\label{eq:shifted main} \partial_t d_i + \nabla \cdot (d_i v_i) = 0, \qquad \nabla d_i = - \sum_{j=1}^n \frac{d_i d_j}{D_{ij}} (v_i - v_j) + \delta \sum_{j=1}^n \frac{1}{D_{ij}}(d_i v_i - d_j v_j)
 	\end{equation} for all $i=1,\dots,n$ and the initial conditions \eqref{eq:inibd} become:
 	\begin{equation}
 		\label{eq:shifted inibd} d_i^0 = c_i^0 + \delta \text{ on } {\T^3}, \ i=1,\dots,n.
 	\end{equation}
 
 	The new variables enjoy the properties:       \begin{equation}\label{bounds}
 	    \delta \le d_i \le 1+\delta \leq 2,
    \end{equation} 
    \begin{equation}\label{sum}
 	    \sum_{i=1}^n d_i = 1+n\delta
    \end{equation} and 
    \begin{equation} \label{estimation:v}
 	\|v_i\|_{L^\infty({\T^3}\times (0,\infty))} = \left\|\frac{c_i u_i} {c_i + \delta }\right\|_{L^\infty}  \le \frac{1}{\delta}\|c u\|_{L^\infty},
 \end{equation}
 	where $\|c u\|_{L^\infty} = \max_{1\le i, j \le n}\{\|c_i u_i\|_{L^\infty},\|\bar{c}_j \bar{u}_j\|_{L^\infty}\}$.
 	
 	Now, we introduce the two matrices
 	\begin{align} 
	\label{sym matrix}
 		A_{ij}(\dd) &= \begin{cases}
 			\sum_{k\ne i} \frac{d_k}{D_{ik}}, & i = j, \\
 			- \frac{\sqrt{d_i d_j}}{D_{ij}}, & i \ne j,
 		\end{cases} 
		\\[5pt]
		\label{pert matrix}
		B_{ij}(\dd) &= \begin{cases}
 		-\sum_{k\ne i} \frac{1}{D_{ik}}, & i = j, \\
 		\frac{\sqrt{ d_j}}{D_{ij} \sqrt{ d_i}}, & i \ne j.
 	\end{cases}
 	\end{align}
 	Note that $A_{i j} (d)$ is symmetric and its spectrum plays a significant role later, while $B_{i j}(d)$ is not symmetric. 
	The linear system in \eqref{eq:shifted main}  can be expressed as 
 	\begin{equation}\label{newsystem}
 		2 \nabla \sqrt{d_i} = - \sum_{j=1}^n \Big(A_{ij}(\dd) + \delta B_{ij}(\dd) \Big) \sqrt{d_j}v_j
 	\end{equation} for all $i=1,\dots,n$. It is a perturbation of \eqref{eq:linear-system} and when $\delta=0$ we recover \eqref{eq:linear-system}. 

    Let $\boldsymbol{z} \in \R^n$. The computation \[ (A(\dd)z)_i=\sum_{j=1}^nA_{ij}(\dd)z_j=\sum_{j\ne i}\frac{d_j}{D_{ij}}z_i-\sum_{j\ne i}\frac{\sqrt{d_id_j}}{D_{ij}}z_j=\sum_{j\not=i}\frac{\sqrt{d_j}}{D_{ij}}(\sqrt{d_j}z_i-\sqrt{d_i}z_j) \] shows that \begin{eqnarray*}
		\operatorname{ran} A(\dd) &=& \left\{ x \in \R^n : \sqrt{\dd} \cdot x = 0 \right\}=:L(\dd) \\ 
		\ker A(\dd) &=& \operatorname{span} \left\{\sqrt{\dd}\right\}=:{L(\dd)}^\perp
	\end{eqnarray*} where by $\sqrt{\dd}$ we denote the vector $(\sqrt{d_1},\dots,\sqrt{d_n})$ and the projection matrices onto $L$ and $L^\perp$ are given by \begin{equation}
		(\mathbb{P}_{L(\dd)})_{ij} = \delta_{ij} - \frac{\sqrt{d_i d_j}}{1+n\delta} \qquad \textnormal{and} \qquad (\mathbb{P}_{L(\dd)^\perp})_{ij} = \frac{\sqrt{d_i d_j}}{1+n\delta}
	\end{equation} respectively.

    Due to \eqref{sum}, \[ \sum_{i=1}^n\frac{d_i}{1+n\delta}=1 \] and using the scaling property $$ A(\dd) = (1+n\delta) A\left(\frac{\dd}{1+n\delta}\right)$$ we can apply the Perron-Frobenius Theorem as in \cite{HuJuTz22}, in order to obtain the estimate:
	\begin{equation} \label{5.7}
		\boldsymbol{z}^\top A(\dd) \boldsymbol{z} = (1+n\delta) \ 	\boldsymbol{z}^\top A\left(\frac{\dd}{1+n\delta}\right) \boldsymbol{z} \ge (1+n\delta)  \mu |\mathbb{P}_{L(\dd)} \boldsymbol{z}|^2
	\end{equation} for all $\boldsymbol{z} \in \R^n$, where $\mu = \min\limits_{i\ne j}\frac{1}{D_{ij}}$.	

%-----------------------------------------------------
%-----------------------------------------------------
%-----------------------------------------------------
%-----------------------------------------------------
	
	\section{Proof of Theorem \ref{theorem-uniqueness}}\label{sec6}

    Using the variables $d_i$ and $v_i$, introduced in Section \ref{sec5}, the relative entropy identity of Corollary \ref{corollary5} reads:
    \begin{eqnarray}
    \label{eq:shifted equality} 
        && F(\dd,\bar{\dd}) \Big |_{t=T} = \sum_{i=1}^n \int_0^T \int_{\T^3} (d_i + \bar{d}_i) \ \nabla(\ln d_i - \ln \bar{d}_i) \cdot (v_i - \bar{v}_i) \dx \dt.
    \end{eqnarray}
	Using the linear system in \eqref{eq:shifted main}, we can rewrite the integrand on the right-hand side of \eqref{eq:shifted equality} as follows:
	\begin{eqnarray*}
		&&\sum_{i=1}^n (d_i + \bar{d}_i)(v_i - \bar{v}_i) \cdot \nabla (\ln d_i - \ln \bar{d}_i) \\ &=& -\sum_{i=1}^n (d_i + \bar{d}_i)(v_i - \bar{v}_i) \cdot \sum_{j\ne i} \frac{1}{D_{ij}}\left( d_j(v_i-v_j) - \bar{d}_j(\bar{v}_i - \bar{v}_j) \right) \\ && + ~ \delta \sum_{i=1}^n (d_i + \bar{d}_i)(v_i - \bar{v}_i) \cdot \sum_{j\ne i} \frac{1}{D_{ij}}\left( v_i - \frac{d_j}{d_i} v_j - \bar{v}_i + \frac{\bar{d}_j}{\bar{d}_i} \bar{v}_j\right) := I_1 + I_2.
	\end{eqnarray*}
	We start with $I_1$: 
    \[ \begin{split}
        I_1 & = -\sum_{i=1}^n d_i (v_i - \bar{v}_i) \cdot \sum_{j\ne i} \frac{1}{D_{ij}}\left( d_j(v_i-v_j) - \bar{d}_j(\bar{v}_i - \bar{v}_j) \right) \\
        & \phantom{xx} - \sum_{i=1}^n \bar{d}_i(v_i - \bar{v}_i) \cdot \sum_{j\ne i} \frac{1}{D_{ij}}\left( d_j(v_i-v_j) - \bar{d}_j(\bar{v}_i - \bar{v}_j) \right)
    \end{split} \] and using the symmetry of $D_{ij}$:
	\begin{eqnarray*}
		I_1 &=&-\sum_{i=1}^n \sum_{j\ne i} \frac{d_i d_j}{2D_{ij}} |(v_i-\bar{v}_i) - (v_j - \bar{v}_j)|^2 - \sum_{i=1}^n \sum_{j\ne i} \frac{d_i}{D_{ij}} (d_j - \bar{d}_j)(v_i - \bar{v}_i)\cdot (\bar{v}_i - \bar{v}_j) \\
		&&-\sum_{i=1}^n \sum_{j\ne i} \frac{\bar{d}_i \bar{d}_j}{2D_{ij}} |(v_i-\bar{v}_i) - (v_j - \bar{v}_j)|^2 - \sum_{i=1}^n \sum_{j\ne i} \frac{ \bar{d}_i}{D_{ij}} (d_j - \bar{d}_j)(v_i - \bar{v}_i)\cdot (v_i - v_j). 
	\end{eqnarray*} 
	Regarding $I_2$, we can split it into two parts:
	\begin{eqnarray*}
		I_2 &=& \delta \sum_{i=1}^n \Big(\sum_{j\ne i} \frac{1}{D_{ij}}\Big) (d_i + \bar{d}_i)|v_i - \bar{v}_i|^2 - \delta \sum_{i=1}^n \sum_{j\ne i} \frac{1}{D_{ij}} (d_i + \bar{d}_i)(v_i - \bar{v}_i) \cdot \left( \frac{d_j}{d_i} v_j - \frac{\bar{d}_j}{\bar{d}_i} \bar{v}_j\right).
	\end{eqnarray*}
	This shows that 
	\begin{equation} \label{5.5}
		\begin{split}
		    & F(\dd,\bar{\dd}) \Big |_{t=T} + \sum_{i=1}^n\sum_{j\ne i} \int_0^T \int_{\T^3}\frac{1}{2D_{ij}}(d_i d_j+\bar{d}_i\bar{d}_j) |(v_i-\bar{v}_i) - (v_j - \bar{v}_j)|^2 \\ 
            & \phantom{xx} = - \sum_{i=1}^n\sum_{j\ne i} \int_0^T \int_{\T^3} \frac{d_i}{D_{ij}} (d_j - \bar{d}_j)(v_i - \bar{v}_i)\cdot (\bar{v}_i - \bar{v}_j)  \ \dx \dt \\ 
            & \phantom{xxx} - \sum_{i=1}^n\sum_{j\ne i} \int_0^T \int_{\T^3} \frac{ \bar{d}_i}{D_{ij}} (d_j - \bar{d}_j)(v_i - \bar{v}_i)\cdot (v_i - v_j)  \ \dx \dt \\  
            & \phantom{xxx} + \delta ~ \sum_{i=1}^n \Big(\sum_{j\ne i} \frac{1}{D_{ij}}\Big) \int_0^T \int_{\T^3} (d_i + \bar{d}_i)|v_i - \bar{v}_i|^2  \ \dx \dt \\ 
            & \phantom{xxx} - \delta ~ \sum_{i=1}^n\sum_{j\ne i} \int_0^T \int_{\T^3} \frac{1}{D_{ij}} (d_i + \bar{d}_i)(v_i - \bar{v}_i) \cdot \left( \frac{d_j}{d_i} v_j - \frac{\bar{d}_j}{\bar{d}_i} \bar{v}_j\right) \ \dx \dt
		\end{split} 
	\end{equation} and we aim to control the terms on the right-hand side by the relative entropy and the dissipation on the left-hand side.
	
	\subsection{Estimation of the dissipation} Following the ideas of \cite{HuJuTz22}, we use the symmetry of $D_{i j}$ to rewrite the dissipation terms as follows: \[ \begin{split}
    \sum_{i=1}^n\sum_{j\ne i} &\frac{d_i d_j}{2D_{ij}}  |(v_i-\bar{v}_i) - (v_j - \bar{v}_j)|^2  =
%        & = \phantom{x} \sum_{i=1}^n\sum_{j\ne i}\frac{d_id_j}{2D_{ij}}((v_i-\bar{v}_i)-(v_j-\bar{v}_j))\cdot(v_i-\bar{v}_i) \\
%    &  \phantom{xx} + \sum_{i=1}^n\sum_{j\ne i}\frac{d_id_j}{2D_{ij}}((v_i-\bar{v}_i)-(v_j-\bar{v}_j))\cdot(v_i-\bar{v}_i) \\
    \phantom{x} \sum_{i=1}^n\sum_{j\ne i}\frac{d_id_j}{D_{ij}}((v_i-\bar{v}_i)-(v_j-\bar{v}_j))\cdot(v_i-\bar{v}_i) \\
    & = \phantom{x} \sum_{i=1}^n\left[\sum_{j\ne i}\frac{d_id_j}{D_{ij}}|v_i-\bar{v}_i|^2-\sum_{j\ne i}\frac{d_id_j}{D_{ij}}(v_j-\bar{v}_j)\cdot(v_i-\bar{v}_i)\right] \\
    & = \phantom{x} \sum_{i=1}^n\left[\sum_{j\ne i}\frac{d_j}{D_{ij}}\sqrt{d_i}(v_i-\bar{v}_i)\cdot\sqrt{d_i}(v_i-\bar{v}_i)-\sum_{j\ne i}\frac{\sqrt{d_id_j}}{D_{ij}}\sqrt{d_i}(v_i-\bar{v}_i)\cdot\sqrt{d_j}(v_j-\bar{v}_j)\right] \\
    & = \phantom{x} \sum_{i=1}^n\sum_{j=1}^nA_{ij}(\dd)\sqrt{d_i}(v_i-\bar{v}_i)\cdot\sqrt{d_j}(v_j-\bar{v}_j).
\end{split} \]
where $A_{ij}(d)$ is the matrix in \eqref{sym matrix}.
Let $\YY = (Y_1,\dots,Y_n)$, where $Y_i = \sqrt{d_i} (v_i - \bar{v}_i)$. By \eqref{5.7}, \[ \begin{split}
    \sum_{i=1}^n\sum_{j=1}^nA_{ij}(\dd)\sqrt{d_i}(v_i-\bar{v}_i)\cdot\sqrt{d_j}(v_j-\bar{v}_j) \geq (1+n\delta)\mu|\mathbb{P}_L(\dd)\boldsymbol{Y}|^2\geq \mu|\mathbb{P}_L(\dd)\boldsymbol{Y}|^2,
\end{split} \] where \[ \begin{split}
    |\mathbb{P}_L(\dd)\YY|^2 & = |\YY|^2-|\mathbb{P}_{L^\perp}(\dd)\YY|^2
\end{split} \] with \[ |\YY|^2=\sum_{i=1}^nd_i|v_i-\bar{v}_i|^2 \] and \[ \begin{split}
    (\mathbb{P}_{L^\perp}(\dd)\YY)_i & = \sum_{j=1}^n\frac{\sqrt{d_id_j}}{1+n\delta}\sqrt{d_j}(v_j-\bar{v}_j) = \frac{\sqrt{d_i}}{1+n\delta}\sum_{j=1}^n(d_jv_j-d_j\bar{v}_j) \\
    & = \frac{\sqrt{d_i}}{1+n\delta}\sum_{j=1}^n(\bar{d}_j\bar{v}_j-d_j\bar{v}_j) = \frac{\sqrt{d_i}}{1+n\delta}\sum_{j=1}^n(\bar{d}_j-d_j)\bar{v}_j
\end{split} \] where we have used the fact that $\sum_{i=1}^nd_iv_i=\sum_{i=1}^n\bar{d}_i\bar{v}_i=0$. 

Then, by \eqref{estimation:v}, \[ \begin{split}
    |\mathbb{P}_{L^\perp}(\dd)Y|^2 & = \frac{1}{1+n\delta}\left(\sum_{j=1}^n(\bar{d}_j-d_j)\bar{v}_j\right)^2\leq\frac{n}{1+n\delta}\sum_{j=1}^n|d_j-\bar{d}_j|^2|\bar{v}_j|^2 \\
    & \leq n\|\bar{v}_j\|_{L^\infty}^2\sum_{j=1}^n|d_j-\bar{d}_j|^2\leq \frac{n}{\delta^2}\|cu\|_{L^\infty}^2 \sum_{j=1}^n|d_j-\bar{d}_j|^2.
\end{split} \] Therefore \[ \sum_{i=1}^n\sum_{j\ne i} \frac{d_i d_j}{2D_{ij}} |(v_i-\bar{v}_i) - (v_j - \bar{v}_j)|^2 \geq \mu\sum_{i=1}^nd_i|v_i-\bar{v}_i|^2-\frac{\mu n}{\delta^2}\|cu\|_{L^\infty}^2 \sum_{j=1}^n|d_j-\bar{d}_j|^2 \] and similarly \[ \sum_{i=1}^n\sum_{j\ne i} \frac{\bar{d}_i \bar{d}_j}{2D_{ij}} |(v_i-\bar{v}_i) - (v_j - \bar{v}_j)|^2 \geq \mu\sum_{i=1}^n\bar{d}_i|v_i-\bar{v}_i|^2-\frac{\mu n}{\delta^2}\|cu\|_{L^\infty}^2 \sum_{j=1}^n|d_j-\bar{d}_j|^2. \] 

Putting everything together, we find:

	\begin{equation} \label{ineq}
{\allowdisplaybreaks
	    \begin{split}
		  & F(\dd,\bar{\dd}) \Big |_{t=T} + \mu 
            \sum_{i=1}^n \int_0^T \int_{\T^3} (d_i+ \bar{d}_i) |v_i - \bar{v}_i|^2 \ \dx \dt \\
            & \le \frac{2 n \mu}{\delta^2}\| c u \|^2_{L^\infty} \sum_{i=1}^n \int_0^T \int_{\T^3} |d_i - \bar{d}_i|^2 \ \dx \dt \\
            & \phantom{x} + \int_0^T \int_{\T^3}\underbrace{-\sum_{i=1}^n\sum_{j\ne i} \frac{d_i}{D_{ij}} (d_j - \bar{d}_j)(v_i - \bar{v}_i)\cdot (\bar{v}_i - \bar{v}_j)}_{J_1}  \ \dx \dt \\ 
            & \phantom{x} + \int_0^T \int_{\T^3}\underbrace{-\sum_{i=1}^n\sum_{j\ne i}  \frac{ \bar{d}_i}{D_{ij}} (d_j - \bar{d}_j)(v_i - \bar{v}_i)\cdot (v_i - v_j)}_{J_2}  \ \dx \dt \\  
            & \phantom{x} + \int_0^T \int_{\T^3} \underbrace{ \delta ~ \sum_{i=1}^n \Big(\sum_{j\ne i} \frac{1}{D_{ij}}\Big)  (d_i + \bar{d}_i)|v_i - \bar{v}_i|^2}_{J_3}  \ \dx \dt \\
            & \phantom{x} + \int_0^T \int_{\T^3} \underbrace{- \delta \sum_{i=1}^n\sum_{j\ne i}  \frac{1}{D_{ij}} (d_i + \bar{d}_i)(v_i - \bar{v}_i) \cdot \left( \frac{d_j}{d_i} v_j - \frac{\bar{d}_j}{\bar{d}_i} \bar{v}_j\right)}_{J_4}  \ \dx \dt
	\end{split}
	}
	\end{equation} and we need to estimate the error terms $J_1,J_2,J_3$ and $J_4$.

 \subsection{Estimation of the error terms} \label{sec6.1} We start with $J_3$. Setting $M = \max\limits_{i\ne j}\frac{1}{D_{ij}}$, we get \begin{equation}\label{J_3}
     J_3 \leq n\delta M \sum_{i=1}^n (d_i + \bar{d}_i)|v_i - \bar{v}_i|^2 .
 \end{equation}
 
 Regarding the terms $J_1$ and $J_2$, we have: \[ \begin{split}
    J_1 & \leq \sum_{i=1}^n\sum_{j\ne i}\frac{d_i}{D_{ij}}|d_j-\bar{d}_j||v_i-\bar{v}_i||\bar{v}_i-\bar{v}_j| \\
    & = \sum_{i=1}^n\sqrt{d_i}|v_i-\bar{v}_i|\sum_{j\ne i}\frac{\sqrt{d_i}}{D_{ij}}|d_j-\bar{d}_j||\bar{v}_i-\bar{v}_j| \end{split} \] and using Young's and Jensen's inequality
    \[ \begin{split} 
    J_1 & \leq \frac{\mu}{4}\sum_{i=1}^nd_i|v_i-\bar{v}_i|^2+\frac{1}{\mu}\sum_{i=1}^n\left(\sum_{j\ne i}\frac{\sqrt{d_i}}{D_{ij}}|d_j-\bar{d}_j||\bar{v}_i-\bar{v}_j|\right)^2 \\
    & \leq \frac{\mu}{4}\sum_{i=1}^nd_i|v_i-\bar{v}_i|^2+\frac{8n^2M^2}{\mu}\|v\|_{L^\infty}^2\sum_{j=1}^n|d_j-\bar{d}_j|^2, \end{split} \] where $\|v\|_{L^\infty}:=\max_{1\leq i\leq n}\|v_i\|_{L^\infty}$. Finally, by \eqref{estimation:v}:
    \[ \begin{split}
    J_1 & \leq \frac{\mu}{4}\sum_{i=1}^nd_i|v_i-\bar{v}_i|^2+\frac{8n^2M^2}{\mu\delta^2}\|cu\|_{L^\infty}^2\sum_{j=1}^n|d_j-\bar{d}_j|^2
\end{split} \] and similarly, \[ J_2 \leq \frac{\mu}{4}\sum_{i=1}^n\bar{d}_i|v_i-\bar{v}_i|^2+\frac{8n^2M^2}{\mu\delta^2}\|cu\|_{L^\infty}^2\sum_{j=1}^n|d_j-\bar{d}_j|^2. \] Hence, we see that \begin{equation}\label{J_1+J_2}
    J_1+J_2 \leq \frac{\mu}{4}\sum_{i=1}^n(d_i+\bar{d}_i)|v_i-\bar{v}_i|^2+\frac{C_1}{\delta^2}\sum_{j=1}^n|d_j-\bar{d}_j|^2
\end{equation} and $C_1>0$ depends on $n$, $\mu$, $M$, and $\|cu\|_{L^\infty}$, but not on $\delta$.

For $J_4$, we first split it into two parts \[ \begin{split}
	J_4 & \leq \delta M \sum_{i=1}^n\sum_{j \ne i}(d_i + \bar{d}_i)|v_i - \bar{v}_i| \left| \frac{d_j}{d_i} v_j - \frac{\bar{d}_j}{\bar{d}_i} \bar{v}_j\right| \\
%     & \leq \delta M \sum_{i=1}^n\sum_{j \ne i} d_i |v_i - \bar{v}_i| \left| \frac{d_j}{d_i} v_j - \frac{\bar{d}_j}{\bar{d}_i} \bar{v}_j\right| + \delta M \sum_{i=1}^n\sum_{j \ne i}  \bar{d}_i |v_i - \bar{v}_i| \left| \frac{d_j}{d_i} v_j - \frac{\bar{d}_j}{\bar{d}_i} \bar{v}_j\right| 
 %    \\
      & \le \delta M \sum_{i=1}^n\sum_{j \ne i} d_i |v_i - \bar{v}_i| \big( \left| \frac{d_j}{d_i} v_j - \frac{d_j}{d_i} \bar{v}_j \right| + \left| \frac{d_j}{d_i} \bar{v}_j - \frac{\bar{d}_j}{\bar{d}_i} \bar{v}_j\right|\big) \\
     & \phantom{xx} + \delta M \sum_{i=1}^n\sum_{j \ne i}  \bar{d}_i |v_i - \bar{v}_i| \big( \left| \frac{d_j}{d_i} v_j - \frac{\bar{d}_j}{\bar{d}_i} v_j \right| + \left| \frac{\bar{d}_j}{\bar{d}_i} v_j- \frac{\bar{d}_j}{\bar{d}_i} \bar{v}_j\right| \big)
\\
    & = \delta M \sum_{i=1}^n\sum_{j \ne i} d_j |v_i - \bar{v}_i| \left| v_j - \bar{v}_j \right| + \delta M \sum_{i=1}^n\sum_{j \ne i} d_i |v_i - \bar{v}_i| |\bar{v}_j| \left| \frac{d_j}{d_i}  - \frac{\bar{d}_j}{\bar{d}_i} \right| \\ 
    & \phantom{xx} + \delta M \sum_{i=1}^n\sum_{j \ne i}  \bar{d}_i |v_i - \bar{v}_i| |v_j| \left| \frac{d_j}{d_i} - \frac{\bar{d}_j}{\bar{d}_i} \right| + \delta M\sum_{i=1}^n\sum_{j \ne i}  \bar{d}_j |v_i - \bar{v}_i| \left| v_j- \bar{v}_j\right| \\ 
    & = \delta M \sum_{i=1}^n\sum_{j \ne i} (d_j + \bar{d}_j) |v_i - \bar{v}_i| \left| v_j - \bar{v}_j \right| \\
    & \phantom{xx} + \delta M \sum_{i=1}^n\sum_{j \ne i} (d_i |\bar{v}_j| + \bar{d}_i |v_j| ) |v_i - \bar{v}_i| \left| \frac{d_j}{d_i} - \frac{\bar{d}_j}{\bar{d}_i} \right| 
    \\
    &=: J_4^1 + J_4^2.
\end{split} \]

The first term is handled by Young's inequality:
\begin{eqnarray*}
	J_4^1 &=& \sum_{j=1}^n\sqrt{d_j+\bar{d}_j}|v_j-\bar{v}_j|\left(\delta M\sqrt{d_j+\bar{d}_j}\sum_{i=1}^n|v_i-\bar{v}_i|\right) \\
    &\le& \frac{\mu}{4} \sum_{j=1}^n(d_j + \bar{d}_j) \left| v_j - \bar{v}_j \right|^2+\frac{4n\delta^2M^2}{\mu} \left(\sum_{i=1}^n |v_i - \bar{v}_i|\right)^2 
\end{eqnarray*} 
then Jensen's inequality in the second term: 
\[ J_4^1 \le \frac{\mu}{4} \sum_{j=1}^n(d_j + \bar{d}_j) \left| v_j - \bar{v}_j \right|^2+\frac{4n^2\delta^2M^2}{\mu} \sum_{i=1}^n |v_i - \bar{v}_i|^2 \] 
and finally we multiply and divide by $d_i+\bar{d}_i$:
\begin{eqnarray*}
    J_4^1 &\leq& \frac{\mu}{4} \sum_{j=1}^n(d_j + \bar{d}_j) \left| v_j - \bar{v}_j \right|^2+\frac{4n^2\delta^2M^2}{\mu} \sum_{i=1}^n (d_i+\bar{d}_i)|v_i - \bar{v}_i|^2\frac{1}{d_i+\bar{d}_i}  \\
    &\le& \frac{\mu}{4} \sum_{j=1}^n(d_j + \bar{d}_j) \left| v_j - \bar{v}_j \right|^2+\frac{2n^2\delta M^2}{\mu} \sum_{i=1}^n (d_i+\bar{d}_i)|v_i - \bar{v}_i|^2 
\end{eqnarray*} where in the last inequality we used that  $d_i+\bar{d}_i\geq2\delta$.

The second term is estimated in a similar way:
\[ \begin{split}
	J_4^2 & \leq \delta M\max\{\|\bar{v}_j\|_{L^\infty},\|v_j\|_{L^\infty}\} \sum_{i=1}^n\sum_{j=1}^n (d_i+\bar{d}_i) |v_i - \bar{v}_i| \left| \frac{d_j}{d_i}  - \frac{\bar{d}_j}{\bar{d}_i} \right| \\
    & \overset{\eqref{estimation:v}}{\le} M\|cu\|_{L^\infty} \sum_{i=1}^n\sum_{j=1}^n (d_i + \bar{d}_i ) |v_i - \bar{v}_i| \left| \frac{d_j}{d_i}  - \frac{\bar{d}_j}{\bar{d}_i} \right| 
    \end{split} \]
and using triangle inequality
\[ \begin{split} 
    J_4^2 & \le M\|cu\|_{L^\infty} \sum_{i=1}^n\sum_{j=1}^n (d_i + \bar{d}_i ) |v_i - \bar{v}_i| \Big(\left| \frac{d_j}{d_i}  - \frac{d_j}{\bar{d}_i} \right| + \left| \frac{d_j}{\bar{d}_i}  - \frac{\bar{d}_j}{\bar{d}_i} \right| \Big) \\ 
    & = M\|cu\|_{L^\infty} \sum_{i=1}^n\sum_{j=1}^n (d_i + \bar{d}_i ) |v_i - \bar{v}_i| \Big( \frac{d_j}{d_i\bar{d}_i} \left| d_i  - \bar{d}_i \right| + \frac{1}{\bar{d}_i}\left| d_j  - \bar{d}_j \right| \Big) \\
    & \le M\|cu\|_{L^\infty} \sum_{i=1}^n\sum_{j=1}^n (d_i + \bar{d}_i ) |v_i - \bar{v}_i| \Big( \frac{2}{\delta^2} \left| d_i  - \bar{d}_i \right| + \frac{1}{\delta}\left| d_j  - \bar{d}_j \right| \Big)
\end{split} \] because $\delta\leq d_i,\bar{d}_i \leq2$ and thus $\frac{1}{d_i\bar{d}_i}=\frac{1}{(c_i+\delta)(\bar{c}_i+\delta)}\leq\frac{1}{\delta^2}$. Then,
\[ \begin{split}
    J_4^2 & = \frac{2M\|cu\|_{L^\infty}}{\delta^2} \sum_{i=1}^n\sum_{j=1}^n (d_i + \bar{d}_i ) |v_i - \bar{v}_i| \left| d_i  - \bar{d}_i \right| \\
    & \phantom{xx} + \frac{M\|cu\|_{L^\infty}}{\delta} \sum_{i=1}^n\sum_{j=1}^n (d_i + \bar{d}_i ) |v_i - \bar{v}_i| \left| d_j  - \bar{d}_j \right|=:J_4^{2,1}+J_4^{2,2}  \, .
    \end{split} \]
Using Young's inequality the first term gives
 \[ \begin{split} 
J_4^{2,1} & = \frac{2nM\|cu\|_{L^\infty}}{\delta^2}\sum_{i=1}^n(d_i + \bar{d}_i ) |v_i - \bar{v}_i| \left| d_i  - \bar{d}_i \right| \\
& = \sum_{i=1}^n\sqrt{d_i+\bar{d}_i}|v_i-\bar{v}_i|\frac{2nM\|cu\|_{L^\infty}}{\delta^2}\sqrt{d_i+\bar{d}_i}|d_i-\bar{d}_i| \\
& \leq \frac{\mu}{8}\sum_{i=1}^n(d_i+\bar{d}_i)|v_i-\bar{v}_i|^2+\frac{16n^2M^2\|cu\|_{L^\infty}^2}{\mu\delta^4}\sum_{i=1}^n|d_i-\bar{d}_i|^2 \, .
\end{split} \] 
Similarly, Young's and Jensen's inequalities give
\[ \begin{split}
    J_4^{2,2} & = \sum_{i=1}^n\sqrt{d_i+\bar{d}_i}|v_i-\bar{v}_i|\frac{M\|cu\|_{L^\infty}}{\delta}\sqrt{d_i+\bar{d}_i}\sum_{j=1}^n|d_j-\bar{d}_j| \\
    & \leq \frac{\mu}{8}\sum_{i=1}^n(d_i+\bar{d}_i)|v_i-\bar{v}_i|^2+\frac{2M^2\|cu\|_{L^\infty}^2}{\mu\delta^2} \sum_{i=1}^n(d_i+\bar{d}_i)\left(\sum_{j=1}^n|d_j-\bar{d}_j|\right)^2 \\
    & \leq \frac{\mu}{8}\sum_{i=1}^n(d_i+\bar{d}_i)|v_i-\bar{v}_i|^2+\frac{8M^2n^2\|cu\|_{L^\infty}^2}{\mu\delta^2}\sum_{j=1}^n|d_j-\bar{d}_j|^2.
\end{split} \] 
Summarizing,
\begin{equation}\label{J_4}
    J_4 \leq \left(\frac{\mu}{2}+C_2\delta\right)\sum_{i=1}^n(d_i+\bar{d}_i)|v_i-\bar{v}_i|^2 + \frac{C_3}{\delta^4}\sum_{j=1}^n|d_j-\bar{d}_j|^2    
\end{equation} where $C_2>0$ depends only on $n,\mu$ and $M$ and $C_3>0$ on $n,\mu,M$ and $\|cu\|_{L^\infty}^2$ and neither of them depends on $\delta$.

Finally, putting together \eqref{ineq}, \eqref{J_3}, \eqref{J_1+J_2} and \eqref{J_4}, we obtain the differential inequality
\[ F(\dd,\bar{\dd}) \Big |_{t=T} + \left(\frac{\mu}{4}-C_4\delta\right) \sum_{i=1}^n \int_0^T \int_{\T^3} (d_i+ \bar{d}_i) |v_i - \bar{v}_i|^2 \ \dx \dt \leq \frac{C_5}{\delta^4}\sum_{i=1}^n\int_0^T \int_{\T^3}|d_i-\bar{d}_i|^2 \ \dx \dt \]
where $C_4$ and $C_5$ are positive constants independent of $\delta$. 

Next, by selecting $0<\delta<\min\{1,\frac{\mu}{4C_4}\}$, we obtain
\begin{equation}\label{final}
    F(\dd,\bar{\dd})\Big |_{t=T} \leq \frac{C_5}{\delta^4}\sum_{i=1}^n\int_0^T \int_{\T^3}|d_i-\bar{d}_i|^2 \ \dx \dt,
\end{equation} where \[ F(\dd,\bar{\dd}) = \sum_{i=1}^n\int_{\T^3}(\ln d_i-\ln\bar{d}_i)(d_i-\bar{d}_i)\dx. \] By Taylor's theorem, \[ (\ln d_i-\ln\bar{d}_i)(d_i-\bar{d}_i) = \frac{1}{\xi}|d_i-\bar{d}_i|^2, \] for some $\xi \in (\min\{d_i,\bar{d}_i\}, \max\{d_i,\bar{d}_i\})\subset(\delta, 1+\delta)$, i.e. \begin{equation} \label{taylor}
    (\ln d_i-\ln\bar{d}_i)(d_i-\bar{d}_i) \geq \frac{1}{1+\delta}|d_i-\bar{d}_i|^2.
\end{equation}

Therefore, \eqref{final} and \eqref{taylor} imply that \[ \sum_{i=1}^n \int_{\T^3}|d_i-\bar{d}_i|^2 \ \dx \Big |_{t = T} \leq C_5\frac{1+\delta}{\delta^4}\sum_{i=1}^n \int_0^T\int_{\T^3}|d_i-\bar{d}_i|^2 \ \dx \dt \] and a direct application of Gr\"onwall's Lemma gives \[ \sum_{i=1}^n \int_{\T^3}|d_i-\bar{d}_i|^2 \ \dx \Big |_{t = T} \leq 0 \] i.e. $d_i=\bar{d}_i$ and thus $c_i=\bar{c}_i$, for all $i=1,\dots,n$. This completes the proof.

%-----------------------------------------------------
%-----------------------------------------------------
%-----------------------------------------------------
%-----------------------------------------------------

\section{Discussion of the hypotheses for uniqueness} \label{sec:discussion}

We discuss here the role of the assumption $c_iu_i\in L^\infty$ that is necessary for the uniqueness result. Refering to  \eqref{ineq}, note that the terms $J_1, J_2$ and $J_4$ need to be estimated in such a way that they produce a term that is dominated by the integral of the symmetrized relative entropy and a term that is absorbed by the dissipation on the left-hand side. The former term should be a multiple of $\sum_i\int_0^T\int_{\T^3}|d_i-\bar{d}_i|^2\dx\dt$, as \eqref{taylor} suggests, while the latter should be a multiple of the dissipation. To do so, we need to apply Young's inequality in each of $J_1,J_2,J_4$ and this requires the differences $v_i-v_j$ and $\bar{v}_i-\bar{v}_j$ to be bounded (which follows if $c_iu_i$ and $\bar{c}_i\bar{u}_i$ are in $L^\infty_{x,t}$). In this sense, the assumption  $c_iu_i\in L^\infty$ is necessary to exploit the full capacity of the relative entropy identity.

The assumption $c_iu_i\in L^\infty_{x,t}$, $i = 1, \dots , n$, is equivalent to $\nabla c_i\in L^\infty_{x,t}$,  $i = 1, \dots , n$. One direction of this implication
is immediate from \eqref{eq:linear-system} and the fact that $0\leq c_i\leq 1$.
For the converse, assume that $c_1,\dots,c_m$ are nonzero and $c_{m+1},\dots,c_n$ are zero (for this we might need to relabel them). By inverting the Maxwell-Stefan system \eqref{eq:linear-system}-\eqref{eq:constraint} (see \cite[Prop. 9, 11]{Gi91}) $c_iu_i = \sum_{k=1}^mA_{ik}\nabla c_k$ for all $i\in\{1,\dots,m\}$ and $c_iu_i = -B_i\nabla c_i$ for $i\in\{m+1,\dots,n\}$, where $A_{ik}$ and $B_i$ are possibly zero but well-defined coefficients that depend on $D_{ij}$ and the positive mass fractions.

We recall that the regularity offered by the existence theory of \cite{JuSt13} is $\nabla \sqrt{c_i} \in L^2_{x,t}$, which implies $\nabla c_i\in L^2_{x,t}$ due to $c_i$ being bounded and is  less than what is required for uniqueness.

\smallskip

{\bf Acknowledgments.} Research partially supported by  King Abdullah University of Science and Technology (KAUST) baseline funds. The first author acknowledges partial support from the Austrian Science Fund (FWF), grants P33010 and F65. This work has received funding from the European Research Council (ERC) under the European Union's Horizon 2020 research and innovation programme, ERC Advanced Grant no. 101018153.

\smallskip

{\bf Conflict of interest.} The authors have no conflict of interest to report.

\end{document}